\documentclass{article}
\usepackage[utf8]{inputenc}
\usepackage[a4paper,margin=2.5cm]{geometry}
\usepackage{graphicx}
\usepackage{amsfonts, amsmath, amssymb}
\usepackage{tikz}
\usetikzlibrary{patterns,angles,quotes,arrows}
\usepackage{enumerate}

\usepackage{hyperref}

\usepackage{float} 
\usepackage{caption}
\usepackage{url}
\usepackage{mathtools} 
\usepackage{colortbl} 
\usepackage{array} 
\usepackage{fancyhdr}
\usepackage[makeroom]{cancel}
\usepackage{amsthm}
\usepackage{subcaption}
\usepackage[numbers]{natbib}
\usepackage{cite}
\newtheoremstyle{mystyle}
  {}
  {}
  {}
  {}
  {\bfseries}
  {:}
  { }
  {}

\newtheorem{theorem}{Theorem}
\newtheorem{proposition}{Proposition}
\newtheorem{lemma}{Lemma}
\newtheorem*{definition}{Definition}

\providecommand{\keywords}[1]{\textbf{Keywords:} #1}
\providecommand{\ams}[1]{\textbf{AMS Classification:} #1}

\begin{document}

\title{The Maximax Minimax Quotient Theorem\thanks{This work was supported by an Early Stage Innovations grant from NASA’s Space Technology Research Grants Program, grant no. 80NSSC19K0209. This material is partially based upon work supported by the United States Air Force AFRL/SBRK under contract no. FA864921P0123.}}

\author{Jean-Baptiste Bouvier\thanks{Department of Aerospace Engineering, University of Illinois at Urbana-Champaign, USA, bouvier3@illinois.edu} \and Melkior Ornik \thanks{Department of Aerospace Engineering and Coordinated Science Laboratory,  University of Illinois at Urbana-Champaign, USA, mornik@illinois.edu}}

\maketitle

\begin{abstract}
    We present an optimization problem emerging from optimal control theory and situated at the intersection of fractional programming and linear max-min programming on polytopes.
    A na\"ive solution would require solving four nested, possibly nonlinear, optimization problems.
    Instead, relying on numerous geometric arguments we determine an analytical solution to this problem.
    In the course of proving our main theorem we also establish another optimization result stating that the minimum of a specific minimax optimization is located at a vertex of the constraint set.
\end{abstract}


\ams{49K35;  90C32; 90C47}

\section{Introduction}

The field of fractional programming studies the optimization of a ratio of functions and made its debut in the 1960s with Charnes and Cooper \citep{charnes}. It has since then expanded to more complex and more general problems \citep{phuong}. However, outside of linear fractional programming, very few analytical results are available; the focus has now largely shifted to developing search algorithms \citep{abdel, pardalos_global}. 
In this paper we are interested in a specific fractional optimization problem introduced in \citep{SIAM_CT} and composed of four nested optimization problems. For this reason, a search algorithm would have a high computational cost and would be especially wasteful since an analytical solution exists.

Our ratio of interest features a max-min optimization \citep{pardalos_book} belonging to the setting of semi-infinite programming \citep{semi-infinite_programming}. Because of the infinite number of constraints, it is not possible to immediately apply the classical results of linear max-min theory \citep{max-min_programming} stating that the maximum is attained on the boundary of the constraint set.
Nonetheless, thanks to the specific geometry of our problem we are able to prove a very similar result, first mentioned as Theorem 3.1 in the authors' work \citep{SIAM_CT}. However, its proof is omitted from \citep{SIAM_CT}.

Armed with this preliminary result on max-min programming, we formulate and establish the Maximax Minimax Quotient Theorem. This result concerns the maximization of a ratio of a maximum and a minimax over two polytopes. In the special case where these polytopes are symmetric, this result reduces to Theorem 3.2 of \citep{SIAM_CT}, whose the proof was again omitted for length concerns.

The remainder of this paper is organized as follows.
Section~\ref{sec:prelim} establishes the existence of the Maximax Minimax Quotient and proves a preliminary optimization result.
Section~\ref{sec:Maximax Minimax} states our central theorem and provides its proof.
Section~\ref{sec:lemmas} gathers all the lemmas involved in the proof of the Maximax Minimax Quotient Theorem.
Section~\ref{sec:continuity} justifies the continuity of two maxima functions used during the proof of our main result.
Finally, Section~\ref{sec:example} illustrates the proof of our theorem on a simple example.

\emph{Notation:} 
We use $\partial X$ to denote the boundary of a set $X$ and its interior is denoted $X^\circ := X \backslash \partial X$.
In $\mathbb{R}^n$ we denote the unit sphere with $\mathbb{S} := \{ x \in \mathbb{R}^n : \|x\| = 1\}$ and the ball of radius $\varepsilon$ centered on $x$ with $B_\varepsilon(x) := \big\{ y \in \mathbb{R}^n : \|y - x\| \leq \varepsilon \big\}$.
The scalar product of vectors is denoted by $\langle \cdot, \cdot \rangle$.
For $x \in \mathbb{R}^n$ and $y \in \mathbb{R}^n$ both nonzero we denote as $\widehat{x, y}$ the signed angle from $x$ to $y$ in the 2D plane containing both of them. We take the convention that the angles are positive when going in the clockwise orientation.

\section{Preliminaries}\label{sec:prelim}

\begin{definition}
    A \emph{polytope} in $\mathbb{R}^n$ is a compact intersection of finitely many half-spaces.
\end{definition}
Thus, this work only considers convex polytopes.
If $X$ and $Y$ are two nonempty polytopes in $\mathbb{R}^n$ with $-X \subset Y^\circ$, and $d \in \mathbb{S}$, we define the \emph{Maximax Minimax Quotient} as
\begin{equation}\label{eq:r_(X,Y)}
    r_{X,Y}(d) := \frac{\underset{x\, \in\, X,\ y\, \in\, Y}{\max} \big\{ \|x + y\| : x + y \in \mathbb{R}^+d \big\} }{ \underset{x\, \in\, X}{\min} \big\{ \underset{y\, \in\, Y}{\max} \big\{ \|x + y\| : x + y \in \mathbb{R}^+d \big\} \big\} }.
\end{equation}
The objective of the Maximax Minimax Quotient Theorem is to determine the direction $d$ that maximizes $r_{X,Y}(d)$.
Note that in the numerator of \eqref{eq:r_(X,Y)}, $x$ and $y$ are chosen together to satisfy the constraint $x+y \in \mathbb{R}^+ d$, while in the denominator this constraint only applies to $y$.
Before starting the actual proof of this theorem, we first need to justify the existence of the minimum and the maxima appearing in \eqref{eq:r_(X,Y)}.

\begin{proposition}\label{prop: r_X,Y well-defined}
    Let $X$, $Y$ be two nonempty polytopes in $\mathbb{R}^n$ with $-X \subset Y^\circ$, $\dim Y = n$ and $d \in \mathbb{S}$. Then,
    \begin{enumerate}[(i)]
        \item $\underset{x\, \in\, X,\, y\, \in\, Y}{\max} \big\{ \|x + y\| : x + y \in \mathbb{R}^+d \big\}$ exists,
        \item $\lambda^*(x,d) := \underset{y\, \in\, Y}{\max} \big\{ \|x + y\| : x + y \in \mathbb{R}^+d \big\}$ exists for all $x \in X$,
        \item $\underset{x\, \in\, X}{\min} \big\{ \lambda^*(x,d) \big\}$ exists,
        \item and $\underset{x\, \in\, X}{\min} \big\{\lambda^*(x,d) \big\} > 0$.
    \end{enumerate}

\end{proposition}
\begin{proof}
    (i) Let $S := \big\{ (x, y) \in X \times Y : x + y \in \mathbb{R}^+ d\big\}$. Set $S$ is a closed subset of the compact set $X \times Y$, so $S$ is compact. Since $X$ is nonempty, we take $x \in X$. Using $-X \subset Y$ we have $-x \in Y$ and $x + (-x) = 0 \in \mathbb{R}^+ d$. Then, $(x,-x) \in S$, so $S$ is nonempty. Function $f : S \rightarrow \mathbb{R}$ defined as $f(x,y) := \|x + y\|$ is continuous, so it reaches a maximum over $S$.
       
    \vspace{2mm}
        
    (ii) For $x \in X$ define $S(x) := \big\{ y \in Y : x + y \in \mathbb{R}^+ d \big\}$. Since $S(x)$ is a closed subset of the compact set $Y$, $S(x)$ is compact. Since $-X \subset Y$, we have $-x \in S(x)$ and so $S(x) \neq \emptyset$. Function $f_x : S(x) \rightarrow \mathbb{R}$ defined as $f_x(y) := \|x + y\|$ is continuous, so it reaches a maximum over $S(x)$, i.e., $\lambda^*$ exists. 
    
    \vspace{2mm}    
    
    (iii) For $x \in X$ and $d \in \mathbb{S}$, the argument of $\lambda^*(x,d)$ is uniquely defined as $y^*(x,d) := \lambda^*(x,d)d - x$ since $\|d\| = 1$ and
    \begin{equation}\label{eq:y^*(x,d)}
        y^*(x, d) = \arg \underset{y\, \in\, Y}{\max} \big\{ \|x + y\| : x + y \in \mathbb{R}^+ d \big\}.
    \end{equation}
    Lemma~\ref{lemma: lambda continuous} shows that $\lambda^*$ is continuous in $x$ and $d$, so $y^*$ is also continuous in $x$ and $d$.
    Then, function $f : X \rightarrow \mathbb{R}$ defined as $f(x) := \|x + y^*(x,d)\|$ is continuous, so it reaches a minimum over the compact and nonempty set $X$. 
        
    \vspace{2mm}
    
    (iv) Note that $y^*(x,d) \in \partial Y$ for all $x \in X$. Indeed, assume for contradiction purposes that there exists $\varepsilon > 0$ such that $B_\varepsilon\big(y^*(x,d) \big) \in Y$. We required $\dim Y = n$ to make this ball of full dimension, so that $z := y^*(x,d) + \varepsilon d \in Y$. Then, $x + z = \big(\lambda^*(x,d) + \varepsilon) d \in \mathbb{R}^+ d$ and $\|x + z\| = \lambda^*(x,d) + \varepsilon > \lambda^*(x,d)$ contradicting the optimality of $\lambda^*$.
    Thus, $y^*(x,d) \in \partial Y$. Since $-X \subset Y^\circ$, we have $\|x + y^*(x,d)\| > 0$ for all $x \in X$. 
\end{proof}

Then, with the assumptions of Proposition~\ref{prop: r_X,Y well-defined} the Maximax Minimax Quotient is well-defined.
The proof of our main theorem relies on another optimization result stating that the argument of the minimum in \eqref{eq:r_(X,Y)} lies at a vertex of $X$.

\begin{definition}
    A \emph{vertex} of a set $X \subset \mathbb{R}^n$ is a point $x \in X$ such that if there are $x_1 \in X$, $x_2 \in X$ and $\lambda \in [0,1]$ with $x = \lambda x_1 + (1-\lambda)x_2$, then $x = x_1 = x_2$. 
\end{definition}
With this definition, a vertex of a polytope corresponds to the usual understanding of a vertex of a polytope.

\begin{theorem}\label{thm:minimum on the vertices}
    Let $d \in \mathbb{S}$, $X$ and $Y$ two polytopes of $\mathbb{R}^n$ with $-X \subset Y$ and $\dim Y = n$. Then, there exists a vertex $v$ of $X$ where $\underset{x\, \in\, X}{\min}\big\{ \lambda^*(x,d) \big\}$ is reached.
\end{theorem}
\begin{proof}
    According to Proposition~\ref{prop: r_X,Y well-defined} the minimum of $\lambda^*$ exists. Then, let $x^* \in X$ such that $\lambda^*(x^*, d) = \underset{x\, \in\, X}{\min}\big\{ \lambda^*(x,d) \big\}$, i.e., $\|y^*(x^*) + x^*\| = \underset{x\, \in\, X}{\min}\ \|y^*(x) + x\|$.
    Since $-x^*$ must minimize the distance between itself and $y^*(x^*) \in \partial Y$, with $-X \subset Y$ obviously $x^* \in \partial X$.
    Assume now that $x^*$ is not on a vertex of $\partial X$. Let $S_x$ be the surface of lowest dimension in $\partial X$ such that $x^* \in S_x$ and $\dim S_x \geq 1$.
    
    Let $v$ be a vertex of $S_x$ and $x(\alpha) := x^* + \alpha (v - x^*)$ for $\alpha \in \mathbb{R}$. Notice that $x(0) = x^*$ and $x(1) = v$. Due to the choice of $v$, the convexity of $S_x$ and $x^*$ not being a vertex, there exists $\varepsilon > 0$ such that $x(\alpha) \in S_x$ for all $\alpha \in [-\varepsilon, 1]$.
    We also define the lengths $L(\alpha) := \|y^*\big(x(\alpha)\big) + x(\alpha)\|$ and $L^* := L(0)$.
    
    Since $\|d\| = 1$ and $y^*\big(x(\alpha)\big) + x(\alpha) \in \mathbb{R}^+ d$, we have $L(\alpha) = \langle y^*\big(x(\alpha)\big) + x(\alpha), d \rangle$.
    By definition of $x^*$, we know that $L^* \leq L(\alpha)$ for all $\alpha \in [-\varepsilon, 1]$. For contradiction purposes assume that there exists $\alpha_0 \in (0, 1]$ such that $L^* < L(\alpha_0)$. We introduce the convexity coefficient $\beta := \frac{\alpha_0}{\alpha_0 + \varepsilon} > 0$ and then
    \begin{align*}
        L^* &= \beta L^* + (1 - \beta)L^* < \beta L(-\varepsilon) + (1-\beta)L(\alpha_0) \\
        &< \beta \langle y^*\big(x(-\varepsilon)\big) + x(-\varepsilon), d \rangle + (1-\beta) \langle y^*\big(x(\alpha_0)\big) + x(\alpha_0), d \rangle  = \langle z + x^*, d\rangle,
    \end{align*}
    with $z := \beta y^*\big(x(-\varepsilon)\big) + (1-\beta) y^*\big(x(\alpha_0)\big)$. Indeed, note that $\beta x(-\varepsilon) + (1-\beta) x(\alpha_0) = x^*$, and $z + x^* \in \mathbb{R}^+ d$.
    Note that $L^* = \underset{y\, \in\, Y}{\max} \big\{ \langle x^* + y, d\rangle : x^* + y \in \mathbb{R}^+ d \big\}$, but $L^* < \langle x^* + z, d\rangle$.
    Given that $z \in Y$ by convexity of $Y$ and $x^* + z \in \mathbb{R}^+ d$, we have reached a contradiction.
    Thus, there is no $\alpha_0 \in (0,1]$ such that $L^* < L(\alpha_0)$. Therefore, for all $\alpha \in [0,1]$, $L(\alpha) = L^*$. By taking $\alpha = 1$, we have $x(\alpha) = v$, so the minimum $L^*$ is also reached on the vertex $v$ of $X$. 
\end{proof}

We have now all the preliminary results necessary to state our central theorem.

\section{The Maximax Minimax Quotient Theorem}\label{sec:Maximax Minimax}

\begin{theorem}[Maximax Minimax Quotient Theorem]\label{thm:varying x_M and x_N}
    If $X$ and $Y$ are two polytopes in $\mathbb{R}^n$ with $-X \subset Y^\circ$, $\dim X = 1$, $\partial X = \{x_1, x_2\}$ with $x_2 \neq 0$ and $\dim Y = n$, then $\underset{d\, \in\, \mathbb{S}}{\max}\ r_{X,Y}(d) = \max \big\{ r_{X,Y}(x_2), r_{X,Y}(-x_2) \big\}$.
\end{theorem}
\begin{proof}
    Since $\dim X = 1$, its extremities $x_1$ and $x_2$ are different, so at least one of them is nonzero. Then, imposing $x_2 \neq 0$ does not restrain the generality of our result.
    Following Proposition~\ref{prop: r_X,Y well-defined}, $r_{X,Y}$ is well-defined. Reusing $y^*$ defined in \eqref{eq:y^*(x,d)}, we introduce $x_M^*(d) := \arg\underset{x\, \in\, X}{\min} \big\{ \|x + y^*(x,d)\| \big\}$ and $x_N^*(d) := \arg \underset{x\, \in\, X}{\max} \big\{ \|x + y^*(x,d)\| : x + y^*(x,d) \in \mathbb{R}^+ d \big\}$. 
    According to Theorem~\ref{thm:minimum on the vertices}, $x_M^*(d) \in \partial X$ for all $d \in \mathbb{S}$ and following Lemma~\ref{lemma: continuity of x_N}, $x_N^*$ is a continuous function of $d$.
    For some $d \in \mathbb{S}$ the $\arg\min$ and $\arg\max$ in the definitions of $x_M^*$ and $x_N^*$ might not be unique; if so we take the arguments ensuring that $x_M^*(d) \in \partial X$ and that $x_N^*$ is continuous.
    We also define $y_N^*(d) := y^*\big( x_N^*(d), d\big)$ and $y_M^*(d) := y^*\big( x_M^*(d), d\big)$. Then,
    \begin{equation*}
        r_{X,Y}(d) = \frac{\underset{y\, \in\, Y}{\max} \big\{ \|y + x_N^*(d)\| : y + x_N^*(d) \in \mathbb{R}^+d \big\} }{ \underset{y\, \in\, Y}{\max} \big\{ \|y + x_M^*(d)\| : y + x_M^*(d) \in \mathbb{R}^+d \big\} } = \frac{\|x_N^*(d) + y_N^*(d)\|}{\|x_M^*(d) + y_M^*(d)\|}.
    \end{equation*}
    
    Since $\dim X = 1$, we can take $\mathcal{P}$ to be a two-dimensional plane containing $X$. Then, we will study how $r_{X,Y}(d)$ varies when $d$ takes values in $\mathbb{S} \cap \mathcal{P}$.
    We introduce the signed angles $\alpha := \widehat{d, \partial Y}$ and $\beta := \widehat{x_2, d}$.  These angles are represented on Figure~\ref{fig:angles illustration} and they take value in $[0, 2\pi)$.
    We parametrize all directions $d \in \mathbb{S} \cap \mathcal{P}$ by the angle $\beta$. Then, we will study how $r_{X,Y}(d)$ varies when $\beta \in [0, 2\pi)$.
    
    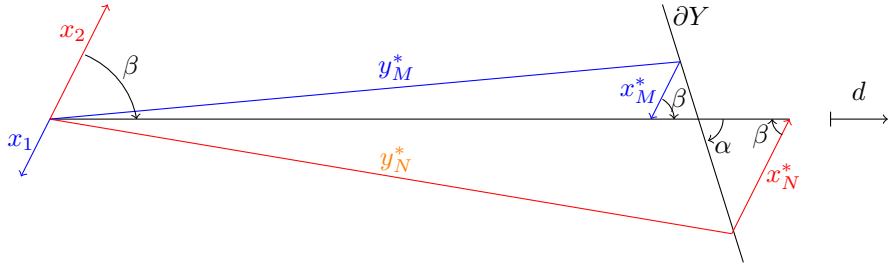
\begin{figure}[htbp!]
        \centering
        \begin{tikzpicture}[scale = 0.76]
            
            \draw[|->] (11.5, 0) -- (12.5, 0);
            \node at (12, 0.5) {$d$};
            \draw (-2, 0) -- (10.8, 0);
            
            \draw (8.6, 2) -- (10, -2.5);
            \node at (9.1, 1.8) {$\partial Y$};
            
            \draw[<-, blue] (-2.5, -1) -- (-2, 0);
            \draw[->, red] (-2, 0) -- (-1, 2);
            \node at (-1.6, 1.5) {\textcolor{red}{$x_2$}};
            \node at (-2.5, -0.4) {\textcolor{blue}{$x_1$}};
            
            \draw[blue] (-2,0) -- (8.9, 1);
            \node at (4, 0.9) {\textcolor{blue}{$y_M^*$}};
            \draw[->, blue] (8.9, 1) -- (8.4, 0);
            \node at (8.2, 0.5) {\textcolor{blue}{$x_M^*$}};
            
            \draw[red] (-2,0) -- (9.8,-2);
            \node at (4,-0.7) {\textcolor{orange}{$y_N^*$}};
            \draw[->, red] (9.8, -2) -- (10.8, 0);
            \node at (10.7, -1) {\textcolor{red}{$x_N^*$}};
            
            \draw[<-] (-0.5, 0) arc (10:67:1.5);
            \node at (-0.6, 0.9) {$\beta$};
            \draw[<-] (8.8, 0) arc (0:60:0.4);
            \node at (8.88, 0.32) {$\beta$};
            \draw[<-] (10.5, 0) arc (180:245:0.3);
            \node at (10.3, -0.3) {$\beta$};
            \draw[->] (9.65, 0) arc (0:-70:0.4);
            \node at (9.65, -0.5) {$\alpha$};
            
        \end{tikzpicture}
        \caption{Illustration of $y_N^*$, $x_N^*$, $y_M^*$ and $x_M^*$ for a direction $d$ parametrized by $\beta$.}
        \label{fig:angles illustration}
    \end{figure}
    
    We first establish in Lemma~\ref{lemma: x_N^*(d) and x_M^*(d) cst on faces} that $x_N^*(d)$ and $x_M^*(d)$ are constant, different and both belong in $\partial X$ when $y_M^*(d)$, $d$ and $y_N^*(d)$ all intersect the same face of $\partial Y$, as illustrated on Figure~\ref{fig:angles illustration}.
    In these situations, Lemma~\ref{lemma:r(d) cst on faces} shows that the ratio $r_{X,Y}$ is constant. Thus, $r_{X,Y}$ can only change when one of the three rays intersects a different face of $\partial Y$ than the other two. We refer to these situations as vertex crossings. Lemma~\ref{lemma: v_pi and v_2pi} introduces the vertices $v_\pi$ and $v_{2\pi}$.
    
    We study the crossing of vertices before $v_\pi$ in Lemma~\ref{lemma: x_N^*(d) and x_M^*(d) cst before v_pi}. During these crossings Lemma~\ref{lemma: crossing before v_pi} shows that $r_{X,Y}$ decreases as $\beta$ increases.
    Lemma~\ref{lemma: crossing v_pi} states that $r_{X,Y}$ reaches a local minimum during the crossing of $v_\pi$.
    As $\beta$ increases between $v_\pi$ and $\pi$, Lemmas~\ref{lemma: x_N^*(d) and x_M^*(d) cst after v_pi} and \ref{lemma: crossing after v_pi} prove that $r_{X,Y}$ increases during the crossing of vertices. Finally, Lemma~\ref{lemma: beta > pi} completes the revolution by showing that $r_{X,Y}$ decreases after $\beta = \pi$ until a local minimum at $v_{2\pi}$ and then increases again until $\beta = 2\pi$. Thus, the directions $d \in \mathcal{P} \cap \mathbb{S}$ maximizing $r_{X,Y}(d)$ are collinear with the set $X$. Note that Figure~\ref{fig:angles illustration} implicitly assumes that $0 \in X$. Lemma~\ref{lemma: 0 notin X} proves that even if $0 \notin X$ all above results still hold.
    Therefore, $\underset{d\, \in\, \mathbb{S}}{\max}\ r_{X,Y}(d) = \underset{\mathcal{P}}{\max} \big\{ \underset{d\, \in\, \mathcal{P}\, \cap\, \mathbb{S} }{\max} r_{X,Y}(d) \big\} = \max\big\{ r_{X,Y}(x_2), r_{X,Y}(-x_2) \big\}$. 
\end{proof}

In the special case where $X$ and $Y$ are symmetric polytopes, this result reduces to Theorem 3.2 of \citep{SIAM_CT}. Indeed, $r_{X,Y}$ becomes an even function which leads to $r_{X,Y}(x_2) = r_{X,Y}(-x_2)$.

\section{Supporting Lemmata}\label{sec:lemmas}

In this section we establish all the lemmas involved in the proof of the Maximax Minimax Quotient Theorem.

\begin{lemma}\label{lemma: x_N^*(d) and x_M^*(d) cst on faces}
    If $d$, $y_N^*(d)$ and $y_M^*(d)$ all intersect the same face of $\partial Y$, then $x_N^*(d)$ and $x_M^*(d)$ are constant, different and both belong to $\partial X$.
\end{lemma}
\begin{proof}
    We introduce the angles $\beta_M := \widehat{x_2, y_M^*}$ and $\beta_N := \widehat{x_2, y_N^*}$.
    Let $\alpha_0$ be the value of $\alpha$ when $\beta = 0$, i.e., when $d$ is positively collinear with $x_2$.
    
    We say that $y_N^*$ is \emph{leading} and $y_M^*$ is \emph{trailing} when $\beta_M < \beta_N$, and conversely when $\beta_N < \beta_M$, we say that $y_M$ is \emph{leading} and $y_N$ is \emph{trailing}.
    
    For each $d \in \mathbb{S} \cap \mathcal{P}$ we define $D(d) := \underset{y\, \in\, Y}{\max} \big\{ \|y\| : y \in \mathbb{R}^+ d \big\}$, whose existence is justified by the compactness of $Y$.

    We say that $y_N^*$ or $y_M^*$ is \emph{outside} when $\|y_N^* + x_N^*\| > D$ or $\|y_M^* + x_M^*\| > D$ respectively. Otherwise, $y_N^*$ or $y_M^*$ is \emph{inside}.
    Directly related to the previous definition, we introduce 
    \begin{equation}\label{eq:delta}
        \delta_M(d) := D(d) - \|x_M^*(d) + y_M^*(d)\|\ \ \text{and} \ \ \delta_N(d) := \|x_N^*(d) + y_N^*(d)\| - D(d).
    \end{equation}
    
    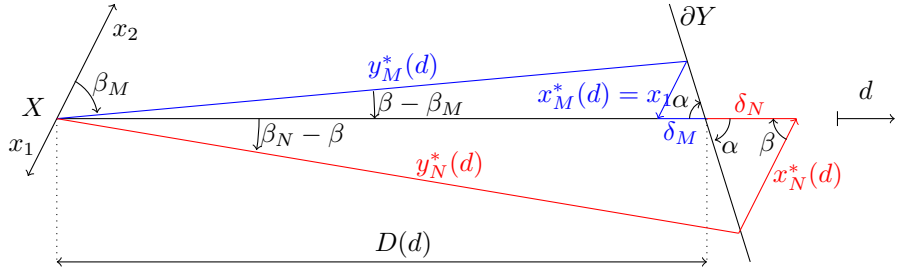
\begin{figure}[htbp!]
        \centering
        \begin{tikzpicture}[scale = 0.76]
            
            \draw[|->] (11.5,0) -- (12.5,0);
            \node at (12, 0.5) {$d$};
            
            \draw (-2, 0) -- (8.4, 0);
            \draw[dotted] (-2, 0) -- (-2, -2.5);
            \draw[dotted] (9.25, 0) -- (9.25, -2.5);
            \draw[<->] (-2, -2.5) -- (9.25, -2.5);
            \node at (4, -2.2) {$D(d)$};
            \draw[red] (9.25, 0) -- (10.8, 0);
            \node at (10, 0.25) {\textcolor{red}{$\delta_N$}};
            \draw[blue] (8.4, 0) -- (9.25, 0);
            \node at (8.8, -0.25) {\textcolor{blue}{$\delta_M$}};
            
             \draw (8.6, 2) -- (10, -2.5);
            \node at (9.1, 1.8) {$\partial Y$};
            
            \node at (-2.4, 0.2) {$X$};
            \draw[->, black] (-2, 0) -- (-1, 2);
            \node at (-0.8, 1.5) {\textcolor{black}{$x_2$}};
            \draw[<-, black] (-2.5, -1) -- (-2, 0);
            \node at (-2.6, -0.5) {\textcolor{black}{$x_1$}};
            
            \draw[blue] (-2,0) -- (8.9, 1);
            \node at (4, 0.9) {\textcolor{blue}{$y_M^*(d)$}};
            \draw[->, blue] (8.9, 1) -- (8.4, 0);
            \node at (7.5, 0.4) {\textcolor{blue}{$x_M^*(d) = x_1$}};
            
            \draw[red] (-2,0) -- (9.8,-2);
            \node at (4.8, -0.8) {\textcolor{red}{$y_N^*(d)$}};
            \draw[->, red] (9.8, -2) -- (10.8, 0);
            \node at (11, -1) {\textcolor{red}{$x_N^*(d)$}};
            
            \draw[->] (8.95, 0) arc (180:110:0.3);
            \node at (8.8, 0.3) {$\alpha$};
            
            \draw[->] (9.65, 0) arc (0:-70:0.4);
            \node at (9.65, -0.5) {$\alpha$};
            \draw[<-] (10.4, 0) arc (180:245:0.4);
            \node at (10.3, -0.45) {$\beta$};
            
             \draw[<-] (-1.3, 0.1) arc (10:60:0.8);
            \node at (-1.05, 0.6) {$\beta_M$};
            \draw[<-] (3.5, 0) arc (0:5:5.5);
            \node at (4.3, 0.25) {$\beta - \beta_M$};
            \draw[->] (1.5, 0) arc (0:-9:3.5);
            \node at (2.25, -0.3) {$\beta_N - \beta$};
            
        \end{tikzpicture}
        \caption{Illustration of $y_N^*(d)$ leading and outside, while $y_M^*(d)$ is trailing and inside the same face of $\partial Y$.}
        \label{fig: x_N^*(d) = -x_M^*(d)}
    \end{figure}
    
    We know from Theorem~\ref{thm:minimum on the vertices} that $x_M^*(d) \in \partial X$ for all $d \in \mathbb{S}$. In the case illustrated on Figure~\ref{fig: x_N^*(d) = -x_M^*(d)}, $x_M^*(d) = x_1$ because it maximizes $\delta_M$. 
    
    If $\alpha + \beta \in \{\pi, 2\pi\}$, then $X$ is parallel with a face of $\partial Y$ making $x_N^*$ and $x_M^*$ not uniquely defined. Regardless, we can still take $x_N^*(d) \neq x_M^*(d)$, with $x_N^*(d) \in \partial X$ and $x_M^*(d) \in \partial X$.
    Otherwise, $x_N^*$ and $x_M^*$ are uniquely defined. Since $x_N^*(d) \in X$, $x_M^*(d) \in X$ for all $d \in \mathbb{S}$ and $\dim X = 1$, vectors $x_N^*(d)$ and $x_M^*(d)$ are always collinear. We then use Thales's theorem and obtain $\delta_N(d) = \delta_M(d) \frac{\|x_N^*(d)\|}{\|x_M^*(d)\|}$. Since $x_N^*(d)$ is chosen to maximize $\delta_N$ and is independent from $\delta_M$, it must have the greatest norm, so $x_N^*(d) \in \partial X$. 
    In the case where $\alpha + \beta \notin \{\pi, 2\pi\}$, $\|x+y\|$ depends on the value of $x$.
    Because $x_N^*(d)$ is chosen to maximize $\|x+y\|$ while $x_M^*(d)$ is minimizing it, we have $x_N^*(d) \neq x_M^*(d)$.
    
    Since $x_N^*$ is continuous according to Lemma~\ref{lemma: continuity of x_N} and $x_N^*(d) \in \big\{x_1, x_2\big\}$, then $x_N^*(d)$ is constant on the faces of $\partial Y$. Because $x_M^*(d) \in \partial X$ too, it must also be constant.  
\end{proof}

\begin{lemma}\label{lemma:r(d) cst on faces}
    When $d$, $y_N^*(d)$ and $y_M^*(d)$ all intersect the same face of $\partial Y$, the ratio $r_{X,Y}(d)$ is constant.
\end{lemma}

\begin{proof}
    Based on Figure \ref{fig: x_N^*(d) = -x_M^*(d)} we apply the sine law in the triangle bounded by $\partial Y$, $\delta_M$ and $x_M^*$
    \begin{equation*}
        \frac{\|x_M^*(d)\|}{\sin \alpha} = \frac{\delta_M(d)}{\sin(\pi - \alpha -\beta)} = \frac{\delta_M(d)}{\sin(\alpha + \beta)}, \hspace{2.4mm} \text{so} \hspace{2.4mm} \frac{\delta_M(d)}{D(d)} = \frac{\|x_M^*(d)\| \sin(\alpha + \beta)}{D(d)\sin \alpha}.
    \end{equation*}
    Similarly for the triangle bounded by $\partial Y$, $\delta_N$ and $x_N^*$, the law of sines yields 
    \begin{equation*}
        \frac{\|x_N^*(d)\|}{\sin \alpha} = \frac{\delta_N(d)}{\sin(\pi - \alpha - \beta)} = \frac{\delta_N(d)}{\sin(\alpha + \beta)}, \hspace{2.5mm} \text{so} \hspace{2.5mm} \frac{\delta_N(d)}{D(d)} = \frac{\|x_N^*(d)\| \sin(\alpha + \beta)}{D(d)\sin \alpha}.
    \end{equation*}
    Even if the two equations above were derived for the specific situation of Figure \ref{fig: x_N^*(d) = -x_M^*(d)}, they hold as long as $y_N^*$, $D$ and $y_M^*$ intersect the same face of $\partial Y$.
    Based on \eqref{eq:delta} we have
    \begin{equation}\label{eq: r_X,Y}
        r_{X,Y}(d) = \frac{D(d) + \delta_N(d)}{D(d) - \delta_M(d)} = \frac{1 + \frac{\delta_N}{D} }{1 - \frac{\delta_M}{D}}.
    \end{equation}

    We will now prove that the ratios $\delta_N / D$ and $\delta_M / D$ do not change on a face of $\partial Y$. Let $d_1 \in \mathcal{P} \cap \mathbb{S}$ and $d_2 \in \mathcal{P} \cap \mathbb{S}$ such that $D(d_1)$, $D(d_2)$, $y_M^*(d_1)$, $y_M^*(d_2)$, $y_N^*(d_1)$ and $y_N^*(d_2)$ all intersect the same face of $\partial Y$, as illustrated on Figure~\ref{fig:ratio constant on faces}. 
    
    \begin{figure}[htbp!]
        \centering
        \begin{tikzpicture}[scale = 0.7]
            
            \draw (8.375, 3) -- (10, -2);
            \node at (9.7, 0.4) {$\partial Y$};
            
            \draw[<-] (-2.25, -0.5) -- (-2, 0);
            \draw[->] (-2, 0) -- (-1.5, 1);
            
            \draw[|->] (10, 2.2) -- (11, 2.42);
            \node at (10.5, 2.7) {$d_1$};
            \draw (-2, 0) -- (8.7, 2);
            \node at (4, 1.5) {$D(d_1)$};
            
            \draw[<-] (-1.2, 0.16) arc (0:65:0.7);
            \node at (-1.1, 0.9) {$\beta_1$};
            \draw[->] (8.3, 1.95) arc (190:120:0.5);
            \node at (8.05, 2.4) {$\alpha_1$};
           
            \draw[|->] (11,-1.7) -- (12,-1.8);
            \node at (11.5, -1.4) {$d_2$};
            \draw (-2, 0) -- (9.8, -1.5);
            \node at (4, -0.4) {$D(d_2)$};
            
            \draw[<-] (-1.5, -0.05) arc (-2:65:0.5);
            \node at (-1.5, -0.4) {$\beta_2$};
            \draw[->] (9.3, -1.4) arc (170:106:0.5);
            \node at (9, -1) {$\alpha_2$};
            
            \draw[<-] (0, -0.2) arc (-15:20:1);
            \node at (1,0) {$\beta_2 - \beta_1$};
            
        \end{tikzpicture}
        \caption{Ratio $r_{X,Y}(d)$ is constant on a face of $\partial Y$.}
        \label{fig:ratio constant on faces}
    \end{figure}
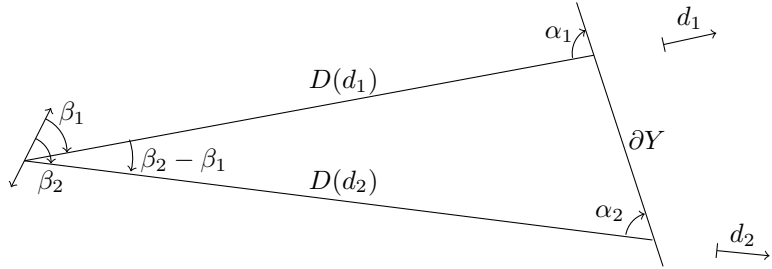

    The sum of the angles of the triangle in Figure~\ref{fig:ratio constant on faces} is
    \begin{equation}\label{eq:alpha + beta = cst on a face}
        (\beta_2 - \beta_1) + \alpha_2 + (\pi - \alpha_1) = \pi \qquad \text{so} \qquad \beta_2 + \alpha_2 = \beta_1 + \alpha_1.
    \end{equation}
    Therefore, $\alpha + \beta$ is constant on faces of $\partial Y$. 
    We use the sine law in the triangle in Figure~\ref{fig:ratio constant on faces} and obtain
    \begin{equation*}
        \frac{D(d_1)}{\sin \alpha_2} = \frac{D(d_2)}{\sin (\pi - \alpha_1)} = \frac{D(d_2)}{\sin \alpha_1}, \quad so,\quad D(d_1)\sin \alpha_1 = D(d_2) \sin \alpha_2.
    \end{equation*}
    According to Lemma~\ref{lemma: x_N^*(d) and x_M^*(d) cst on faces} we also know that $x_N^*(d_1) = x_N^*(d_2)$, thus
    \begin{equation*}
        \frac{\delta_N(d_1)}{D(d_1)} = \frac{\|x_N^*(d_1)\| \sin(\alpha_1 + \beta_1)}{D(d_1) \sin \alpha_1} = \frac{\|x_N^*(d_2)\| \sin(\alpha_2 + \beta_2)}{D(d_2) \sin \alpha_2} = \frac{\delta_N(d_2)}{D(d_2)}.
    \end{equation*}
    The same holds for $\delta_M / D$. Hence, \eqref{eq: r_X,Y} yields $r_{X,Y}(d_1) = r_{X,Y}(d_2)$.  
\end{proof}

\vspace{3mm}

\begin{lemma}\label{lemma: v_pi and v_2pi}
    There are two vertices of $Y \cap \mathcal{P}$, namely $v_\pi$ and $v_{2\pi}$ whose crossing by $d$ makes the angle $\alpha+\beta$ become greater than $\pi$ and $2\pi$ respectively.
\end{lemma}
\begin{proof}
    We have taken the convention that the angles are positively oriented in the clockwise orientation.
    According to \eqref{eq:alpha + beta = cst on a face}, the angle $\alpha + \beta$ is constant on a face of $\partial Y$. When $d$ crosses a vertex of external angle $\varepsilon$ as represented on Figure~\ref{fig: x_v}, the value of $\alpha$ has a discontinuity of $+\varepsilon$. Let $q$ be the number of vertices of $\partial Y$ and $\varepsilon_i$ the external angle of the $i^{th}$ vertex $v_i$. Since $Y \cap \mathcal{P}$ is a polygon, $\sum_{i = 1}^q \varepsilon_i = 2\pi$.
    We can then represent the evolution of $\alpha + \beta$ as a function of $\beta$ with Figure~\ref{fig:graph of alpha + beta}.
    Instead of labeling the horizontal axis with the values taken by $\beta$ as the corresponding vector $d(\beta)$ crosses the vertex $v_i$, we directly use $v_i$ with a slight abuse of notation.
    
    \begin{figure}[htbp!]
        \centering
        \begin{tikzpicture}[scale = 0.7]
            \draw[<->] (0, 5) -- (0, 0) -- (10.5, 0);
            \node at (11, 0) {$\beta$};
            \node at (0, 5.3) {$\alpha + \beta$};
            \node at (0, -0.5) {$0$};
            \draw (10, -0.1) -- (10, 0.1);
            \node at (10, -0.5) {$2\pi$};
            
            \draw (-0.1, 0.4) -- (2, 0.4) -- (2, 0.9) -- (3, 0.9) -- (3, 1.5) -- (5, 1.5) -- (5, 2.5) -- (6, 2.5) -- (6, 2.8) -- (7, 2.8) -- (7, 3.5) -- (8.5, 3.5) -- (8.5, 4.2) -- (10, 4.2);
            
            \draw (2, -0.1) -- (2, 0.1);
            \node at (2, -0.5) {$v_1$};
            \draw (3, -0.1) -- (3, 0.1);
            \node at (3, -0.5) {$v_2$};
            \draw (5, -0.1) -- (5, 0.1);
            \node at (5, -0.5) {$v_3$};
            \draw (6, -0.1) -- (6, 0.1);
            \draw (7, -0.1) -- (7, 0.1);
            \draw (8.5, -0.1) -- (8.5, 0.1);
            \node at (8.5, -0.5) {$v_q$};
            
            \node at (-0.5, 0.4) {$\alpha_0$};
            \draw (-0.1, 0.9) -- (0.1, 0.9);
            \node at (-1, 0.9) {$\alpha_0 + \varepsilon_1$};
            \draw (-0.1, 1.5) -- (0.1, 1.5);
            \node at (-1.6, 1.5) {$\alpha_0 + \varepsilon_1 + \varepsilon_2$};
            \draw (-0.1, 2.5) -- (0.1, 2.5);
            \node at (-2.2, 2.5) {$\alpha_0 + \varepsilon_1 + \varepsilon_2+\varepsilon_3$};
            \draw (-0.1, 2.8) -- (0.1, 2.8);
            \draw (-0.1, 3.5) -- (0.1, 3.5);
            \draw (-0.1, 4.2) -- (0.1, 4.2);
            \node at (-1, 4.2) {$\alpha_0 + 2\pi$};
        
        \end{tikzpicture}
        \caption{Evolution of $\alpha + \beta$ with $\beta$ increasing clockwise in $[0, 2\pi)$.}
        \label{fig:graph of alpha + beta}
    \end{figure}
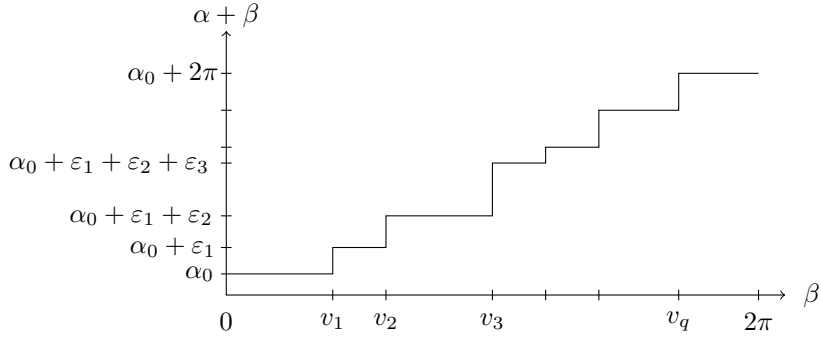
    
    Recall that $\alpha_0$ is the value of $\alpha$ when $\beta = 0$. After a whole revolution $\alpha + \beta = \alpha_0 + 2\pi$. So there are two vertices $v_{\pi}$ and $v_{2\pi}$ where $\alpha + \beta$ first crosses $\pi$ and then $2\pi$. 
    In the eventuality that $\alpha + \beta = \pi$ or $2\pi$ on a face of $\partial Y$, we define $v_\pi$ or $v_{2\pi}$ as the vertex preceding the face.  
\end{proof}

\begin{lemma}\label{lemma: x_N^*(d) and x_M^*(d) cst before v_pi}
    During the crossing of vertices before $v_\pi$ as $\beta$ increases, $x_N^*(d) = x_2$ and $x_M^*(d) = x_1$. They are constant, different and both belong in $\partial X$.
\end{lemma}
\begin{proof}
    We study the crossing of a vertex $v$ of angle $\varepsilon$ between the faces $F_1$ and $F_2$ of $\partial Y$.
    For each vertex $v$ we introduce $x_v$ the vector collinear with $X$, going from $v$ to the ray directed by $d$, as illustrated on Figure~\ref{fig: x_v} and we say that the crossing of $v$ is ongoing as long as $\|x_v\| < \max\{\|x_1\|, \|x_2\|\}$. We also define $\delta_v := \|v + x_v\| - D$.
    
    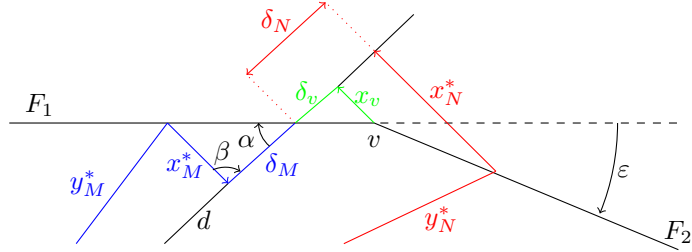
\begin{figure}[htbp!]
        \centering
        \begin{tikzpicture}[scale = 0.8]
            
            \draw (-6, 0) -- (0, 0) -- (5, -2.1);
            \draw[dashed] (0.1, 0) -- (5,0);
            \node at (-5.5, 0.3) {$F_1$};
            \node at (5, -1.8) {$F_2$};
            \node at (0, -0.3) {$v$};
            
            \draw[->] (4,0) arc (0:-22:4);
            \node at (4.1, -0.8) {$\varepsilon$};
            
            \draw (-3.45, -2) -- (-2.4, -1);
            \draw (-0.6, 0.6) -- (0.65, 1.8);
            \node at (-2.8, -1.6) {$d$};
            
            \draw[red] (-0.5, -2) -- (2, -0.8);
            \node at (1.1, -1.6) {\textcolor{red}{$y_N^*$}};
            \draw[->, red] (2, -0.8) -- (0, 1.2);
            \node at (1.2, 0.5) {\textcolor{red}{$x_N^*$}};
            \draw[red, dotted] (-1.3, 0) -- (-2.1, 0.8);
            \draw[red, dotted] (0, 1.2) -- (-0.8, 2);
            \draw[red, <->] (-2.1, 0.8) -- (-0.8, 2);
            \node at (-1.6, 1.7) {\textcolor{red}{$\delta_N$}};
            
            \draw[blue] (-4.9, -2) -- (-3.4, 0);
            \node at (-4.7, -1) {\textcolor{blue}{$y_M^*$}};
            \draw[->, blue] (-3.4, 0) -- (-2.4, -1);
            \node at (-3.1, -0.7) {\textcolor{blue}{$x_M^*$}};
            \draw[blue] (-2.4, -1) -- (-1.3, 0);
            \node at (-1.5, -0.7) {\textcolor{blue}{$\delta_M$}};
           
          \draw[->, green] (0, 0) -- (-0.6, 0.6);
          \node at (-0.1, 0.4) {\textcolor{green}{$x_v$}};
          \draw[green] (-1.3, 0) -- (-0.6, 0.6);
          \node at (-1.1, 0.5) {\textcolor{green}{$\delta_v$}};
           
          \draw[<-] (-2.2, -0.8) arc (60:114:0.5);
          \node at (-2.5, -0.5) {$\beta$};
            
            \draw[<-] (-1.9, 0) arc (180:230:0.5);
            \node at (-2.1, -0.3) {$\alpha$};
            
        \end{tikzpicture}
        \caption{Illustration of $x_v$ during the crossing of a vertex $v$, with $y_N^*$ leading.}
        \label{fig: x_v}
    \end{figure}

    Before starting the crossing of $v_\pi$ we have $\alpha + \beta \in (\alpha_0, \pi)$. This situation is depicted on Figure~\ref{fig: x_N^*(d) = -x_M^*(d)}, where $y_N^*$ is leading and outside, so $y_N^*$ reaches the vertex before $y_M^*$ and $d$. The length of $x_N^*(d)$ can vary to maximize $\delta_N$, so $y_N^*$ could still intersect $F_1$, even if the crossing is ongoing. We have seen in Lemma~\ref{lemma: x_N^*(d) and x_M^*(d) cst on faces} that if $y_N^*$ is still on $F_1$, then it must be the furthest possible to maximize $\delta_N$, in that case $y_N^* = v$. Otherwise, $y_N^*$ intersects $F_2$. We want to establish a criterion to distinguish these two possible scenarios.
    
    We first consider the scenario where $y_N^* = v$ and $x_N^*(d) = x_v$. We take $y \in F_2 \backslash \{v\}$ such that $x_2 + y \in \mathbb{R}^+d$ as represented on Figure~\ref{fig: y_N^* in v} and we define $\delta := \|x_2 + y\| - D$.
    
    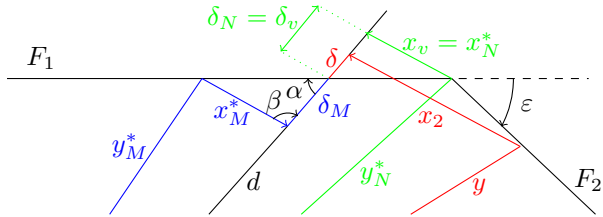
\begin{figure}[htbp!]
        \centering
        \begin{tikzpicture}[scale = 0.9]
            
            \draw (-5, 0) -- (1.5, 0) -- (3.6, -2);
            \draw[dashed] (1.6, 0) -- (3.5,0);
            \node at (-4.5, 0.3) {$F_1$};
            \node at (3.5, -1.5) {$F_2$};
            
            \draw[->] (2.4,0) arc (0:-28:1.5);
            \node at (2.6, -0.4) {$\varepsilon$};
            
            \draw (-2.05, -2) -- (-0.9, -0.7);
            \draw (0, 0.35) -- (0.55, 1);  
            \node at (-1.4, -1.5) {$d$};
            
            \draw[red] (0.9, -2) -- (2.5, -1);
            \node at (1.9, -1.6) {\textcolor{red}{$y$}};
            \draw[->, red] (2.5, -1) -- (0, 0.35);
            \node at (1.2, -0.6) {\textcolor{red}{$x_2$}};
            \draw[red] (-0.3, 0) -- (0, 0.35);
            \node at (-0.25, 0.3) {\textcolor{red}{$\delta$}};
            
            \draw[green] (-0.7, -2) -- (1.5, 0);
            \node at (0.4, -1.4) {\textcolor{green}{$y_N^*$}};
            \draw[green, ->] (1.5, 0) -- (0.25, 0.675);
            \node at (1.5, 0.5) {\textcolor{green}{$x_v = x_N^*$}};
            \draw[green, dotted] (-0.3, 0) -- (-1, 0.4);
            \draw[green, dotted] (0.25, 0.675) -- (-0.45, 1.075);
            \draw[green, <->] (-1, 0.4) -- (-0.45, 1.075);
            \node at (-1.4, 0.9) {\textcolor{green}{$\delta_N = \delta_v$}};
            
            \draw[blue] (-3.5, -2) -- (-2.15, 0);
            \node at (-3.2, -1) {\textcolor{blue}{$y_M^*$}};
            \draw[->, blue] (-2.15, 0) -- (-0.9, -0.7);
            \node at (-1.7, -0.5) {\textcolor{blue}{$x_M^*$}};
            \draw[blue] (-0.9, -0.7) -- (-0.3, 0);
            \node at (-0.2, -0.4) {\textcolor{blue}{$\delta_M$}};
           
            \draw[<-] (-0.75, -0.55) arc (60:135:0.3);
            \node at (-1.1, -0.35) {$\beta$};
            
            \draw[<-] (-0.6, 0) arc (180:230:0.3);
            \node at (-0.8, -0.2) {$\alpha$};
            
        \end{tikzpicture}
        \caption{Illustration of the crossing scenario where $y_N^* = v$.}
        \label{fig: y_N^* in v}
    \end{figure}
    
    Since $\delta_N$ must be maximized by the choice of $y_N^*$ and $y \neq y_N^*$, we have $\delta < \delta_N = \delta_v$. But $\|x_2\| > \|x_v\|$, so the line segment corresponding to $x_2$ crosses the interior of $Y$. Focusing on this part of Figure~\ref{fig: y_N^* in v} we obtain Figure~\ref{fig: zoom x inside}.
    
    \begin{figure}[htbp!]
        \centering
        \begin{tikzpicture}[scale = 0.8]
            
            \draw (-5, 0) -- (1.5, 0) -- (3.6, -2);
            \draw[dashed] (1.7, 0) -- (3.5,0);
            \node at (0.3, 0.25) {$F_1$};
            \node at (3, -1) {$F_2$};
            \node at (1.5, 0.2) {$v$};
            
            \draw[->] (2.4,0) arc (0:-28:1.5);
            \node at (2.6, -0.4) {$\varepsilon$};
            
            \draw (-5, -1) -- (-4, 0);
            \draw (-3, 1) -- (-2.75, 1.25);
            \node at (-4.5, -0.8) {$d$};
            
            \draw[red] (-4, 0) -- (-3, 1);
            \node at (-3.7, 0.6) {\textcolor{red}{$\delta$}};
            
            \draw[red, ->] (3.3, -1.7) -- (-3, 1);
            \node at (0, -0.6) {\textcolor{red}{$x_2$}};
            \draw[red] (2.7, -2) -- (3.3, -1.7);
            
            \draw[<-] (-3.2, 0.8) arc (230:320:0.4);
            \node at (-2.8, 0.4) {$\beta$};
            
            \draw[<-] (-3.6, 0) arc (0:47:0.4);
            \node at (-3.4, 0.2) {$\alpha$};

        \end{tikzpicture}
        \caption{Illustration of the line segment corresponding to $x_2$ crossing the interior of $Y$ in Figure~\ref{fig: y_N^* in v}.}
        \label{fig: zoom x inside}
    \end{figure}
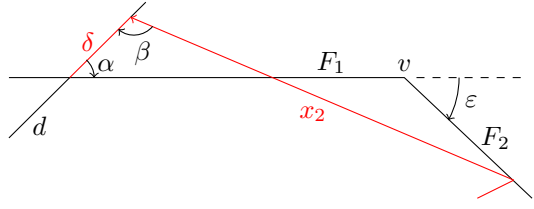
    
    Two of the angles of the triangle delimited by $F_1$, $F_2$ and $x_2$ are $\pi - \alpha - \beta$ and $\pi - \varepsilon$. Therefore, their sum is in $(0, \pi)$ and thus $\alpha + \beta + \varepsilon > \pi$. Since we assumed that $\alpha + \beta \in (\alpha_0, \pi)$, the vertex $v$ must in fact be $v_\pi$ for this scenario to happen.
    
    \vspace{2mm}
    
    Thus, the crossing of a vertex preceding $v_\pi$ follows the second scenario as depicted on Figure~\ref{fig: x_v} with $y_N^* \in F_2$. We study Figure~\ref{fig: zoom x_v x_N^*} which is a more detailed view of Figure~\ref{fig: x_v}, with $\delta_0$ depending solely on $d$ and $\varepsilon$. 
    
    \begin{figure}[htbp!]
        \centering
        \begin{tikzpicture}[scale = 0.7]
            
            \draw (-5, 0) -- (0, 0) -- (2.5, -0.5);
            \draw[dashed] (-2.6, 0.55) -- (0, 0) -- (2.5,0);
            \node at (-4.5, 0.3) {$F_1$};
            \node at (2.9, -0.6) {$F_2$};
            \node at (0, -0.3) {$v$};
            
            \draw[->] (2,0) arc (0:-11:2);
            \node at (2.3, -0.2) {$\varepsilon$};
            
            \draw (-4, -1) -- (-3, 0);
            \node at (-3.3, -0.5) {$d$};
            \draw[blue] (-3, 0) -- (-2.5, 0.5);
            \node at (-3.1, 0.4) {\textcolor{blue}{$\delta_0$}};
            \draw[green] (-2.5, 0.5) -- (-1.5, 1.5);
            \node at (-2.8, 1.2) {\textcolor{green}{$\delta_v - \delta_0$}};
            \draw[red] (-1.5, 1.5) -- (-1, 2);
            \node at (-2.1, 2) {\textcolor{red}{$\delta_N - \delta_v$}};
            \draw (-1, 2) -- (-0.75, 2.25);
            
            \draw[green, ->] (0,0) -- (-1.5, 1.5);
            \node at (-1.1, 0.7) {\textcolor{green}{$x_v$}};
            \draw[red, ->] (1.3, -0.27) -- (-1, 2);
            \node at (0.5, 1.1) {\textcolor{red}{$x_N^*$}};
            
        \end{tikzpicture}
        \caption{Illustration of $x_v$ and $x_N^*$ in Figure~\ref{fig: x_v}.}
        \label{fig: zoom x_v x_N^*}
    \end{figure}
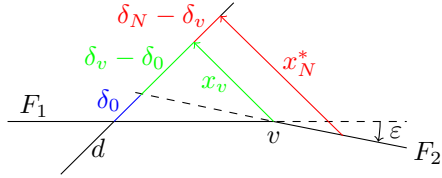
    
    Since $x_v$ and $x_N^*(d)$ are collinear, we can apply Thales's theorem in Figure~\ref{fig: zoom x_v x_N^*} and obtain that $\delta_N - \delta_0 = (\delta_v - \delta_0) \frac{\|x_N^*(d)\|}{\|x_v(d)\|}$. Then, $\delta_N$ is maximized when $\|x_N^*(d)\|$ is maximal, so $x_N^*(d) = x_2$ during the crossing. 
    We know from Theorem~\ref{thm:minimum on the vertices} that $x_M^*(d) \in \partial X$ for all $d \in \mathbb{S}$.
    Then, as in Lemma~\ref{lemma: x_N^*(d) and x_M^*(d) cst on faces}, $x_N^*$ and $x_M^*$ are constant and different since $x_N^*$ is continuous in $d$, so $x_M^*(d) = x_1$.      
\end{proof}

\vspace{3mm}

\begin{lemma}\label{lemma: crossing before v_pi}
    During the crossing of vertices before $v_\pi$ as $\beta$ increases, $r_{X,Y}(d)$ decreases.
\end{lemma}
\begin{proof}
    The leading vector $y_N^*$ is outside and crosses a vertex $v$ between the faces $F_1$ and $F_2$ of $\partial Y$ while $\beta$ increases. We separate the vertex crossing into two parts: when only $y_N^* \in F_2$, and when both $d \in F_2$ and $y_N^* \in F_2$. 
    Let $\varepsilon > 0$ be the external angle of the vertex as shown on Figure~\ref{fig:y_N^* crossing a vertex}. 
    
    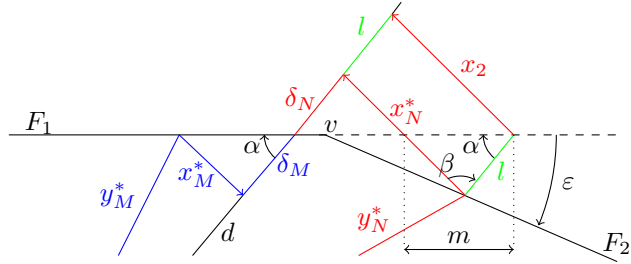
\begin{figure}[htbp!]
        \centering
        \begin{tikzpicture}[scale = 0.8]
            
            \draw (-5, 0) -- (0.2,0) -- (5, -2.1);
            \node at (0.3, 0.15) {$v$};
            \draw[dashed] (0.1, 0) -- (5,0);
            \node at (-4.5, 0.2) {$F_1$};
            \node at (5, -1.8) {$F_2$};
            
            \draw[->] (4,0) arc (0:-22:4);
            \node at (4.2, -0.8) {$\varepsilon$};
            
            \draw (-1.98, -2) -- (-1.15, -1);
            \draw (1.3, 2) -- (1.45, 2.2);
            \node at (-1.4, -1.6) {$d$};
            
            \draw[red] (0.75, -2) -- (2.5, -1);
            \node at (1, -1.4) {\textcolor{red}{$y_N^*$}};
            \draw[->, red] (2.5, -1) -- (0.5, 1);
            \node at (1.55, 0.4) {\textcolor{red}{$x_N^*$}};
            
            \draw[blue] (-3.2, -2) -- (-2.2, 0);
            \node at (-3.2, -1) {\textcolor{blue}{$y_M^*$}};
            \draw[->, blue] (-2.2, 0) -- (-1.15, -1);
            \node at (-1.9, -0.65) {\textcolor{blue}{$x_M^*$}};
            
            \draw[->, red] (3.3, 0) -- (1.3, 2);
            \node at (2.65, 1.1) {\textcolor{red}{$x_2$}};
            \draw[green] (2.5, -1) -- (3.3, 0);
            \node at (3.1, -0.6) {\textcolor{green}{$l$}};
            \draw[green] (0.5, 1) -- (1.3, 2);
            \node at (0.8, 1.8) {\textcolor{green}{$l$}};
           
            \draw[<-] (2.67, -0.75) arc (60:114:0.5);
            \node at (2.2, -0.5) {$\beta$};
            
            \draw[<-] (-0.8, 0) arc (180:230:0.5);
            \node at (-1, -0.2) {$\alpha$};
            \draw[blue] (-1.15, -1) -- (-0.3, 0);
            \node at (-0.3, -0.5) {\textcolor{blue}{$\delta_M$}};
            \draw[red] (-0.3, 0) -- (0.5, 1);
            \node at (-0.2, 0.6) {\textcolor{red}{$\delta_N$}};
            
            \draw[<-] (2.8, 0) arc (180:230:0.5);
            \node at (2.6, -0.2) {$\alpha$};
            
            \draw[dotted] (1.5, 0) -- (1.5, -1.9);
            \draw[dotted] (3.3, 0) -- (3.3, -1.9);
            \draw[<->] (1.5, -1.9) -- (3.3, -1.9);
            \node at (2.4, -1.7) {$m$};
            
        \end{tikzpicture}
        \caption{Part I of the crossing of vertex $v$ by $y_N^*$ leading and outside as $\beta$ increases.}
        \label{fig:y_N^* crossing a vertex}
    \end{figure}
    
    According to Lemma~\ref{lemma:r(d) cst on faces}, $r_{X,Y}$ is constant on faces of $\partial Y$ and we call $r_{F_1}$ its value on the face $F_1$.
    If $F_1$ was prolonged past $v$ with a straight line (dashed line on Figure~\ref{fig:y_N^* crossing a vertex}), then we would have $y_N^*(d) \in F_1$ and $r_{X,Y}(d) = r_{F_1}$. 
    But, $y_N^*(d) \in F_2$ as proven in Lemma~\ref{lemma: x_N^*(d) and x_M^*(d) cst before v_pi} because the crossing occurs before $v_\pi$. We call $l$ the resulting difference in $\delta_N$ as illustrated on Figure~\ref{fig:y_N^* crossing a vertex}. Notice that the two green segments of length $l$ in Figure~\ref{fig:y_N^* crossing a vertex} are parallel. We parametrize the position of $y_N^*$ on $F_2$ with the length $m$ as defined on Figure~\ref{fig:y_N^* crossing a vertex}. When $y_N^* = v$, $m = 0$, and $m$ increases with $\beta$. Using the sine law we obtain
    \begin{equation}\label{eq:loss crossing part 1}
        \frac{m}{\sin \beta} = \frac{l}{\sin (\pi - \alpha - \beta)} = \frac{l}{\sin(\alpha + \beta)}.
    \end{equation}
    Then,
    \begin{equation}\label{eq:r(d) r_F_1}
        r_{X,Y}(d) = \frac{D + \delta_N}{D - \delta_M} = \frac{D + \delta_N + l}{D - \delta_M} - \frac{l}{D - \delta_M} = r_{F_1} - \frac{m \sin (\alpha + \beta)}{(D - \delta_M) \sin(\beta)}.
    \end{equation}
    By definition the length $m$ is positive. Since $-x_M^* \in Y^\circ$ but $y_M^* \in \partial Y$, we have $D - \delta_M = \|y_M^* + x_M^* \| > 0$. Before $v_\pi$ we have $\alpha + \beta \in (\alpha_0, \pi)$. In that case $\sin(\alpha + \beta) > 0$ and $\sin(\beta) > 0$.
    Therefore, the term subtracted from $r_{F_1}$ is positive, i.e., $r_{X,Y}(d) < r_{F_1}$. 
    
    \vspace{2mm}
    
    We can now tackle the second part of the crossing, when $y_N^*$ and $d$ both have crossed the vertex as illustrated on Figure~\ref{fig:y_N^* and d crossing a vertex}.
    
    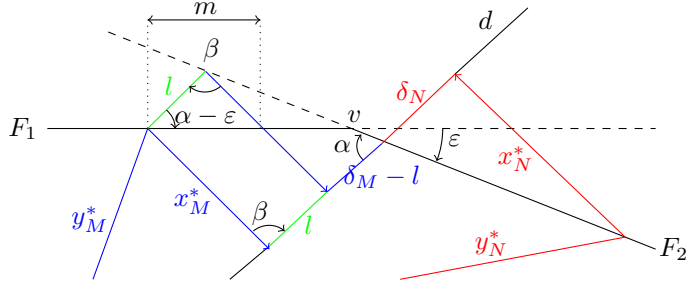
\begin{figure}[htbp!]
        \centering
        \begin{tikzpicture}[scale = 0.8]
            
            \draw (-5, 0) -- (0,0) -- (5, -2);
            \node at (0.05, 0.2) {$v$};
            \draw[dashed] (-4, 1.6) -- (0,0) -- (5,0);
            \node at (-5.4, 0) {$F_1$};
            \node at (5.3, -2) {$F_2$};
            
            \draw[->] (1.5,0) arc (0:-21:1.5);
            \node at (1.7, -0.2) {$\varepsilon$};
            
            \draw (-2, -2.5) -- (-1.35, -1.95);
            \draw (1.7, 0.9) -- (2.9, 2);
            \node at (2.2, 1.8) {$d$};
            
            \draw[red] (0.8, -2.5) -- (4.5, -1.8);
            \node at (2.3, -1.9) {\textcolor{red}{$y_N^*$}};
            \draw[->, red] (4.5, -1.8) -- (1.7, 0.9);
            \node at (2.7, -0.5) {\textcolor{red}{$x_N^*$}};
            
            \draw[blue] (-4.25, -2.5) -- (-3.35, 0);
            \node at (-4.3, -1.5) {\textcolor{blue}{$y_M^*$}};
            \draw[->, blue] (-3.35, 0) -- (-1.35, -2);
            \node at (-2.6, -1.2) {\textcolor{blue}{$x_M^*$}};
            
            \draw[->, blue] (-2.4, 0.95) -- (-0.4, -1.05);
            \draw[green] (-3.35, 0) -- (-2.4, 0.95);
            \node at (-3, 0.7) {\textcolor{green}{$l$}};
            
             \draw[green] (-1.35, -1.95) -- (-0.4, -1.05);
             \node at (-0.7, -1.6) {\textcolor{green}{$l$}};
             
           
            \draw[<-] (-2.9, 0) arc (0:50:0.4);
            \node at (-2.4, 0.2) {$\alpha - \varepsilon$};
            \draw[->] (-2.15, 0.7) arc (-50:-130:0.4);
            \node at (-2.3, 1.3) {$\beta$};
            \draw[<-] (-1.1, -1.7) arc (50:130:0.4);
            \node at (-1.5, -1.4) {$\beta$};
            
            \draw[<-] (0.16, -0.1) arc (140:230:0.3);
            \node at (-0.15, -0.3) {$\alpha$};
            \draw[blue] (0.5, -0.25) -- (-0.4, -1.05);
            \node at (0.5, -0.8) {\textcolor{blue}{$\delta_M - l$}};
            \draw[red] (0.5, -0.25) -- (1.7, 0.9);
            \node at (1, 0.6) {\textcolor{red}{$\delta_N$}};
            
            \draw[dotted] (-3.35, 0) -- (-3.35, 1.8);
            \draw[dotted] (-1.5, 0) -- (-1.5, 1.8);
            \draw[<->] (-3.35, 1.8) -- (-1.5, 1.8);
            \node at (-2.4, 2) {$m$};
        \end{tikzpicture}
        \caption{Part II of the crossing of vertex $v$ by $y_N^*$ leading and outside as $\beta$ increases.}
        \label{fig:y_N^* and d crossing a vertex}
    \end{figure}
    
    If $F_2$ was prolonged with a straight line before $v$ and $y_M^* \in F_2$, then we would have $r_{X,Y}(d) = r_{F_2}$, value of $r_{X,Y}$ on $F_2$. But that is not the case, $y_M^*(d) \in F_1$ and the resulting difference in $\delta_M$ is called $l$. Using the sine law in Figure~\ref{fig:y_N^* and d crossing a vertex}, we can relate $l$ to $m$
    \begin{equation}\label{eq:gain crossing part 2}
         \frac{m}{\sin \beta} = \frac{l}{\sin(\pi - \beta - \alpha + \varepsilon)} = \frac{l}{\sin(\alpha + \beta - \varepsilon)}.
    \end{equation}
    We have $\alpha + \beta \in (\alpha_0, \pi)$, so $\sin(\beta) > 0$.
    If $\alpha$ was still measured between $d$ and $F_1$, then its value would be $\alpha_{F_1} = \alpha - \varepsilon$. Since we are before the crossing of $v_\pi$, $\alpha_{F_1} + \beta \in (\alpha_0, \pi)$, i.e., $\alpha + \beta - \varepsilon \in (\alpha_0, \pi)$.
    This yields $\sin(\alpha + \beta - \varepsilon) > 0$, which makes $l > 0$, because the length $m$ is positive by definition.
    Then,
    \begin{equation}\label{eq:r(d) r_F_2}
         r_{F_2} = \frac{D + \delta_N}{D - (\delta_M - l)} = \frac{D + \delta_N}{D - \delta_M + l} < \frac{D + \delta_N}{D - \delta_M} = r_{X,Y}(d).
    \end{equation}
    Thus, the ratio $r_{X,Y}$ decreases during the crossing of a vertex before $v_\pi$.  
\end{proof}

\vspace{3mm}

\begin{lemma}\label{lemma: crossing v_pi}
    During the crossing of $v_\pi$, the ratio $r_{X,Y}(d)$ reaches a local minimum.
\end{lemma}
\begin{proof}
    Recall that before the crossing, $x_N^*(d) = x_2$ and $x_M^*(d) = x_1$.
    During the crossing of $v_\pi$, i.e., when $\| x_{v_\pi} \| < \max \{ \|x_1\|, \|x_2\| \}$, we have $\alpha + \beta \leq \pi$ but $\alpha + \beta + \varepsilon > \pi$. The situation is illustrated on Figure~\ref{fig: y_N^* crossing v_pi}. We showed in Lemma~\ref{lemma: x_N^*(d) and x_M^*(d) cst before v_pi} that $y_N^* = v_\pi$ and $x_N^*(d) = x_{v_\pi}$.
    
    \begin{figure}[htbp!]
        \centering
        \begin{tikzpicture}[scale = 0.8]
            
            \draw (-5, 0) -- (1.5, 0) -- (3.6, -2);
            \draw[dashed] (1.7, 0) -- (3.5,0);
            \node at (-4.5, 0.3) {$F_1$};
            \node at (3.6, -1.5) {$F_2$};
            
            \draw[->] (2.4,0) arc (0:-28:1.5);
            \node at (2.6, -0.4) {$\varepsilon$};
            
            \draw (-2.05, -2) -- (-1.5, -1.35);
            \draw (0.75, 1.25) -- (1, 1.5);
            \node at (-1.6, -1.7) {$d$};
            
            \draw[red] (-0.7, -2) -- (1.5, 0);
            \node at (0.8, -1.1) {\textcolor{red}{$y_N^*$}};
            \draw[red, ->] (1.5, 0) -- (0.25, 0.675);
            \node at (1.1, 0.5) {\textcolor{red}{$x_{v_\pi}$}};
            \draw[red] (-0.3, 0) -- (0.25, 0.675);
            \node at (-1, 0.5) {\textcolor{red}{$\delta_N = \delta_{v_\pi}$}};
            
            \draw[blue] (-4.5, -2) -- (-4, 0);
            \node at (-4.6, -1) {\textcolor{blue}{$y_M^*$}};
            \draw[->, blue] (-4, 0) -- (-1.5, -1.35);
            \node at (-3, -0.9) {\textcolor{blue}{$x_M^*$}};
            \draw[blue] (-1.5, -1.35) -- (-0.3, 0);
            \node at (-0.5, -0.8) {\textcolor{blue}{$\delta_M$}};
           
          \draw[<-] (-1.25, -1.1) arc (60:130:0.5);
          \node at (-1.7, -0.8) {$\beta$};
            
            \draw[<-] (-0.8, 0) arc (180:230:0.5);
            \node at (-1, -0.2) {$\alpha$};
            
            \draw[->, green] (3, 0) -- (0.75, 1.25);
            \node at (2, 0.8) {\textcolor{green}{$x_2$}};
            \draw[green] (0.25, 0.675) -- (0.75, 1.25);
            \node at (0.45, 1.2) {\textcolor{green}{$l$}};
            
        \end{tikzpicture}
        \caption{Crossing of $v_\pi$, with $y_N^* = v_\pi$.}
        \label{fig: y_N^* crossing v_pi}
    \end{figure}
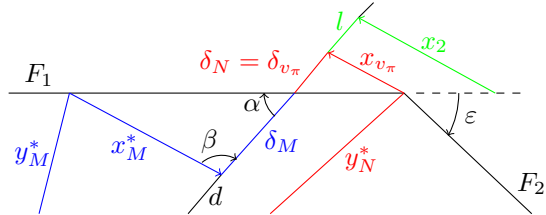
    
    If $F_1$ was prolonged with a straight line (dashed line of Figure~\ref{fig: y_N^* crossing v_pi}), we would have $y_N^* \neq v_\pi$, $x_N^*(d) = x_2$ and the ratio would be $r_{F_1} = \frac{D + \delta_{v_\pi} + l}{D - \delta_M}$, which is the value of $r_{X,Y}$ on $F_1$. Since $d$ has not yet crossed $v_\pi$, $\alpha + \beta < \pi$ and thus \eqref{eq:loss crossing part 1} and \eqref{eq:r(d) r_F_1} still hold, leading to $r_{X,Y}(d) < r_{F_1}$.
    
    \vspace{2mm}
    
    Once $d$ has crossed $v_\pi$, we still have $y_N^* = v_\pi$ to maximize $\delta_N$. Then, the equality $x_N^*(d) = x_{v_\pi}$ holds during the whole crossing, i.e., as $x_{v_\pi}$ goes from $x_2$ to $x_1$. The second part of the crossing is illustrated on Figure~\ref{fig: d passed v_pi}.
    
    \begin{figure}[htbp!]
        \centering
        \begin{tikzpicture}[scale = 0.9]
            
            \draw (-5, -1.15) -- (0, 0) -- (3.4, -2.65);
            \draw[dashed] (0, 0) -- (-1.4, 1);
            \node at (-1.8, -0.2) {$F_1$};
            \node at (1.9, -1.2) {$F_2$};
            
            \draw (-2.7, -2.65) -- (-0.95, -1.5);
            \draw (1.6, 0.25) -- (2, 0.5);
            \node at (2.1, 0.4) {$d$};
            
            \draw[blue, ->] (-3.75, -0.85) -- (-0.95, -1.5);
            \node at (-2.4, -0.9) {\textcolor{blue}{$x_M^* = x_1$}};
            \draw[blue] (-4.5, -2.5) -- (-3.75, -0.85);
            \node at (-3.8, -1.8) {\textcolor{blue}{$y_M^*$}};
            \draw[blue] (-0.95, -1.5) -- (0.6, -0.45);
            \node at (0.2, -1.1) {\textcolor{blue}{$\delta_M$}};
            
            \draw[blue, ->] (3.2, -2.5) -- (-0.95, -1.5);
            \node at (1, -1.8) {\textcolor{blue}{$x_2$}};
            \draw[blue] (2.2, -2.65) -- (3.2, -2.5);
            
            \draw[red, ->] (0, 0) -- (0.9, -0.22);
            \node at (0.55, 0) {\textcolor{red}{$x_{v_\pi}$}};
            \draw[red] (-3.4, -2.65) -- (0, 0);
            \node at (-0.5, -0.65) {\textcolor{red}{$y_N^*$}};
            \draw[red] (0.6, -0.45) -- (0.9, -0.22);
            \node at (1, -0.5) {\textcolor{red}{$\delta_N$}};
            
            \draw[green, ->] (-1.2, 0.9) -- (1.6, 0.25);
            \node at (0.4, 0.7) {\textcolor{green}{$x_1$}};
            \draw[green] (0.9, -0.22) -- (1.6, 0.25);
            \node at (1.4, -0.15) {\textcolor{green}{$l$}};
            
        \end{tikzpicture}
        \caption{Illustration of the endpoint of $y_M^*$ switching from $F_1$ to $F_2$ during the crossing of $v_\pi$.}
        \label{fig: d passed v_pi}
    \end{figure}
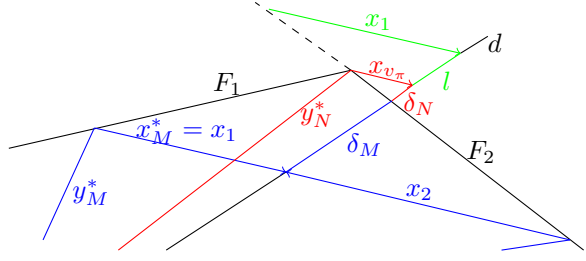
    
    Assume that during the entire crossing of $v_\pi$, $x_M^*(d) = x_1$. Then, at the end of the crossing we will have $y_M^* = v_\pi$ and $x_M^*(d) = x_{v_\pi} = x_N^*(d)$, which contradicts the definitions of $x_M^*(d)$ and $x_N^*(d)$, they must be different. Thus, $x_M^*(d)$ does not remain equal to $x_1$ during the entire crossing. Since $x_M^* \in \big\{x_1, x_2\big\}$, at some point $x_M^*$ switches to $x_2$ as $y_M^*$ switches from $F_1$ to $F_2$. This switching point is illustrated on Figure~\ref{fig: d passed v_pi}, and $y_M^*$ becomes the leading vector.
    
    After this switch, $y_M^* \in F_2$ and $x_M^*(d) = x_2$. If $F_2$ was prolonged with the dashed line on Figure~\ref{fig: d passed v_pi}, we would have $x_N^* = x_1$ instead of $x_{v_\pi}$ with a gain of $l$ for $\delta_N$ making the ratio equal to $r_{F_2} = \frac{D + \delta_N + l}{D - \delta_M}$, value of $r_{X,Y}$ on $F_2$. But $x_N^* = x_{v_\pi}$ and $l > 0$, thus $r_{F_2} > \frac{D + \delta_N}{D - \delta_M} = r_{X,Y}(d)$. Therefore, $r_{X,Y}$ reaches a local minimum during the crossing of $v_\pi$.  
\end{proof}

\vspace{3mm}

\begin{lemma}\label{lemma: x_N^*(d) and x_M^*(d) cst after v_pi}
  During the crossing of vertices after $v_\pi$ as $\beta$ increases until $\pi$, $x_N^*(d) = x_1$ and $x_M^*(d) = x_2$. They are constant different and both in $\partial X$.
\end{lemma}
\begin{proof}
    After the crossing of $v_\pi$, $\alpha + \beta \in (\pi, \alpha_0 + \pi)$ and $y_M^*$ is leading and inside as established in Lemma~\ref{lemma: crossing v_pi}. Thus, $y_M^*$ is the first to reach vertex $v$. Since $x_M^* \in \{x_1, x_2\}$ we cannot have $x_M^* = x_v$ during the entire crossing because $x_v$ is a continuous function of $\beta$. Thus $y_M^*$ passes $v$ and belongs to $F_2$. 
    In Lemma~\ref{lemma: continuity of x_N} we showed that $x_N^*$ is continuous in $d$. Thus, $x_N^*(d)$ cannot switch like $x_M^*(d)$ did around $v_\pi$ to take the lead. Instead, $x_N^*(d)$ is trailing as illustrated on Figure~\ref{fig: crossing after v_pi}.
    
    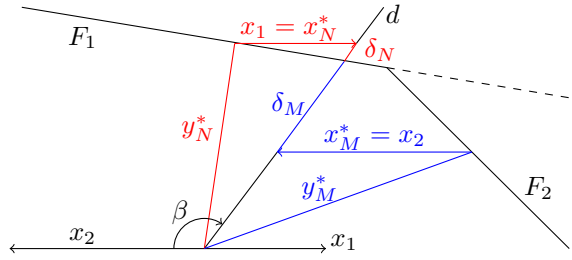
\begin{figure}[htbp!]
        \centering
        \begin{tikzpicture}[scale = 0.8]
            
            \draw[->] (0, 0) -- (-3.2, 0);
            \node at (-2, 0.2) {$x_2$};
            \draw[->] (0, 0) -- (2, 0);
            \node at (2.3, 0.1) {$x_1$};
            
            \draw (-3, 4) -- (3, 3) -- (6, 0);
            \node at (-2, 3.5) {$F_1$};
            \node at (5.5, 1) {$F_2$};
            \draw[dashed] (3, 3) -- (6, 2.5);
            
            \draw (0, 0) -- (1.2, 1.6);
            \draw (2.5, 3.4) -- (2.95, 4);
            \node at (3.1, 3.9) {$d$};
            
            \draw[red] (0, 0) -- (0.5, 3.4);
            \node at (-0.1, 2) {\textcolor{red}{$y_N^*$}};
            \draw[red, ->] (0.5, 3.4) -- (2.5, 3.4);
            \node at (1.4, 3.65) {\textcolor{red}{$x_1 = x_N^*$}};
            \draw[red] (2.5, 3.4) -- (2.3, 3.1);
            \node at (2.9, 3.3) {\textcolor{red}{$\delta_N$}};
            
            \draw[blue] (0, 0) -- (4.4, 1.6);
            \node at (1.9, 1) {\textcolor{blue}{$y_M^*$}};
            \draw[blue, ->] (4.4, 1.6) -- (1.2, 1.6);
            \node at (2.8, 1.85) {\textcolor{blue}{$x_M^* = x_2$}};
            \draw[blue] (1.2, 1.6) -- (2.3, 3.1);
            \node at (1.4, 2.4) {\textcolor{blue}{$\delta_M$}};
            
            \draw[->] (-0.5, 0) arc (180:54:0.5);
            \node at (-0.4, 0.6) {$\beta$};
            
        \end{tikzpicture}
        \caption{Crossing of a vertex $v$ after $v_\pi$.}
        \label{fig: crossing after v_pi}
    \end{figure}
    
    Since $y_N^* \in F_1$ during the crossing, we can apply Thales's theorem on Figure~\ref{fig: crossing after v_pi} and obtain that for a fixed $d$, $\delta_N$ is proportional to $\|x_N^*(d)\|$. Thus, to maximize $\delta_N$ we have $x_N^*(d) \in \partial X$ and, since $y_N^*$ is trailing, we have $x_N^*(d) = x_1$ during the entire crossing. By the definitions of $x_N^*(d)$ and $x_M^*(d)$, we have $x_N^*(d) \neq x_M^*(d)$. Since both $x_N^*(d)$ and $x_M^*(d)$ belong to $\partial X = \big\{x_1, x_2\big\}$, then $x_M^*(d) = x_2$ during the entire crossing.    
\end{proof}

\vspace{3mm}

\begin{lemma}\label{lemma: crossing after v_pi}
    During the crossing of vertices after $v_\pi$ as $\beta$ increases until $\pi$, $r_{X,Y}(d)$ increases. 
\end{lemma}

\begin{proof}
    The leading vector $y_M^*$ is inside and crosses a vertex $v$ between faces $F_1$ and $F_2$ as $\beta$ increases. 
    We define $\beta' : = \pi - \beta$. 
    Then, reversing the crossing illustrated on Figure~\ref{fig: crossing after v_pi} is exactly the crossing illustrated on Figure~\ref{fig:y_N^* crossing a vertex} with $\beta'$ increasing and $x_1$ and $x_2$ exchanged. According to Lemma~\ref{lemma: crossing before v_pi}, in that reversed crossing $r_{X,Y}$ is decreasing.
    Therefore, $r_{X,Y}$ increases during the crossing of vertices after $v_\pi$ as $\beta$ increases until $\pi$.  
\end{proof}

\vspace{3mm}

\begin{lemma}\label{lemma: beta > pi}
    For $\beta > \pi$, $r_{X,Y}(d)$ decreases until $v_{2\pi}$ where it reaches a local minimum. After $v_{2\pi}$ as $\beta$ increases until $2\pi$, $r_{X,Y}(d)$ increases.
\end{lemma}
\begin{proof}
    Let us change the angle convention, so that angles are now positively oriented in the counterclockwise orientation. The vertex that was previously labeled as $v_{2\pi}$ becomes the new $v_\pi$.
    Then, we only need to apply Lemmas~\ref{lemma: x_N^*(d) and x_M^*(d) cst before v_pi}, \ref{lemma: crossing before v_pi}, \ref{lemma: crossing v_pi}, \ref{lemma: x_N^*(d) and x_M^*(d) cst after v_pi} and \ref{lemma: crossing after v_pi} to this new configuration to conclude the proof.  
\end{proof}

\vspace{3mm}

\begin{lemma}\label{lemma: 0 notin X}
    All above results hold even if $0 \notin X$.
\end{lemma}
\begin{proof}
    In all the figures we made the implicit assumption that $0 \in X$, so that $x_1$ and $x_2$ were negatively collinear.
    Let $x_1$ be positively collinear with $x_2$ and $\|x_2\| > \|x_1\|$.
    
    On Figure~\ref{fig: x_N^*(d) = -x_M^*(d)}, we would now have $y_N^*(d)$ and $y_M^*(d)$ both outside. Then, the definition of $\delta_M$ should be adapted. Let $\delta_M(d) := \|x_M^*(d) + y_M^*(d)\| - D(d)$ and then $r_{X,Y}(d) = \frac{D+ \delta_N}{D + \delta_M}$. Except for this modification, we would still have $x_N^*(d) = x_2$ and $x_M^*(d) = x_1$. Thales theorem can be used similarly to show that $x_N^*(d) \in \partial X$. Therefore, Lemma~\ref{lemma: x_N^*(d) and x_M^*(d) cst on faces} holds.
    
    In the proof of Lemma~\ref{lemma:r(d) cst on faces} we still have $\delta_N / D$ and $\delta_M / D$ invariant with respect to $d$ on a given face of $\partial Y$, so $r_{X,Y}$ is still constant on faces.
    Lemma~\ref{lemma: v_pi and v_2pi} is not affected at all.
    The first part of the crossing of a vertex before $v_\pi$ as $\beta$ increases is illustrated by Figure~\ref{fig: y_N^* crossing with 0 not in X}.
    
    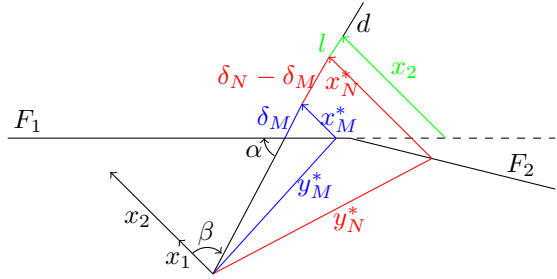
\begin{figure}[htbp!]
        \centering
        \begin{tikzpicture}[scale = 0.9]
            
            \draw[->] (0, 0) -- (-0.5, 0.5);
            \node at (-0.5, 0.2) {$x_1$};
            \draw[->] (-0.5, 0.5) -- (-1.5, 1.5);
            \node at (-1.1, 0.8) {$x_2$};
            
            \draw (-3, 2) -- (2, 2) -- (5, 1.25);
            \node at (-2.7, 2.25) {$F_1$};
            \node at (4.5, 1.6) {$F_2$};
            \draw[dashed] (2.1, 2) -- (5, 2);
            
            \draw[->, blue] (1.8, 2) -- (1.3, 2.5);
            \node at (1.85, 2.27) {\textcolor{blue}{$x_M^*$}};
            \draw[blue] (0, 0) -- (1.8, 2);
            \node at (1.5, 1.3) {\textcolor{blue}{$y_M^*$}};
            \draw[blue] (1.05, 2) -- (1.3, 2.5);
            \node at (0.9, 2.3) {\textcolor{blue}{$\delta_M$}};
            
            \draw[->, red] (3.2, 1.7) -- (1.7, 3.2);
            \node at (1.9, 2.8) {\textcolor{red}{$x_N^*$}};
            \draw[red] (0, 0) -- (3.2, 1.7);
            \node at (2, 0.8) {\textcolor{red}{$y_N^*$}};
            \draw[red] (1.3, 2.5) -- (1.7, 3.2);
            \node at (0.8, 2.9) {\textcolor{red}{$\delta_N - \delta_M$}};
            
            \draw[->, green] (3.4, 2) -- (1.9, 3.5);
            \node at (2.8, 3) {\textcolor{green}{$x_2$}};
            \draw[green] (1.7, 3.2) -- (1.9, 3.5);
            \node at (1.6, 3.4) {\textcolor{green}{$l$}};
            
            \draw (0,0) -- (1.05, 2);
            \draw (1.9, 3.5) -- (2.2, 4);
            \node at (2.2, 3.7) {$d$};
            
            \draw[->] (-0.3, 0.3) arc (135:45:0.3);
            \node at (-0.1, 0.6) {$\beta$};
            
            \draw[<-] (0.75, 2) arc (180:240:0.3);
            \node at (0.6, 1.8) {$\alpha$};
            
        \end{tikzpicture}
        \caption{Part I of the crossing of a vertex before $v_\pi$ with $0 \notin X$.}
        \label{fig: y_N^* crossing with 0 not in X}
    \end{figure}
    
    For $\delta_M$ to be minimized and $\delta_N$ to be maximized, the Thales theorem clearly proves that $x_M^* \in \partial X$ and $x_N^* \in \partial X$ during the crossing. We still have $x_N^*(d) = x_2$ and $x_M^*(d) = x_1$, so Lemma~\ref{lemma: x_N^*(d) and x_M^*(d) cst before v_pi} holds.
    
    Following the reasoning in Lemma~\ref{lemma: crossing before v_pi}, we have $l > 0$, which leads to
    \begin{equation*}
        r_{F_1} = \frac{D + \delta_N + l}{D + \delta_M} > \frac{D + \delta_N}{D + \delta_M} = r_{X,Y}(d).
    \end{equation*}
    During the second part, both $y_N^* \in F_2$ and $y_M^* \in F_2$ but $d \in F_1$. This situation is illustrated on Figure~\ref{fig: y_N^* and y_M^* crossing with 0 not in X}.
    
    \begin{figure}[htbp!]
        \centering
        \begin{tikzpicture}[scale = 0.9]
        
            \draw[->] (-1, 0) -- (-1.6, 0.8);
            \node at (-1.5, 0.2) {$x_1$};
            \draw[->] (-1.6, 0.8) -- (-2.5, 1.9);
            \node at (-2.3, 1.2) {$x_2$};
            
            \draw (-3, 2) -- (1, 2) -- (4, 0.5);
            \node at (-2.7, 2.25) {$F_1$};
            \node at (3.7, 0.9) {$F_2$};
            \draw[dashed] (1.1, 2) -- (4, 2);
            \draw[dashed] (1, 2) -- (-1, 3);

            \draw[->, blue] (1.6, 1.7) -- (1, 2.5);
            \node at (1.6, 2.2) {\textcolor{blue}{$x_M^*$}};
            \draw[blue] (-1, 0) -- (1.6, 1.7);
            \node at (1, 1) {\textcolor{blue}{$y_M^*$}};
            \draw[blue] (0.7, 2.15) -- (1, 2.5);
            \node at (0.4, 2.5) {\textcolor{blue}{$\delta_M - l$}};
            
            \draw[->, red] (3, 1) -- (1.5, 3.15);
            \node at (2.4, 2.4) {\textcolor{red}{$x_N^*$}};
            \draw[red] (-1, 0) -- (3, 1);
            \node at (1.5, 0.3) {\textcolor{red}{$y_N^*$}};
            \draw[red] (1, 2.5) -- (1.5, 3.15);
            \node at (0.7, 3.1) {\textcolor{red}{$\delta_N - \delta_M$}};
            
            \draw[green] (0.6, 2) -- (0.7, 2.15);
            \node at (0.5, 2.15) {\textcolor{green}{$l$}};
            
            \draw (-1,0) -- (0.6, 2);
            \draw (1.5, 3.15) -- (1.75, 3.5);
            \node at (-0.3, 1.3) {$d$};
            
            
            
        \end{tikzpicture}
        \caption{Part II of the crossing of a vertex before $v_\pi$ with $0 \notin X$.}
        \label{fig: y_N^* and y_M^* crossing with 0 not in X}
    \end{figure}
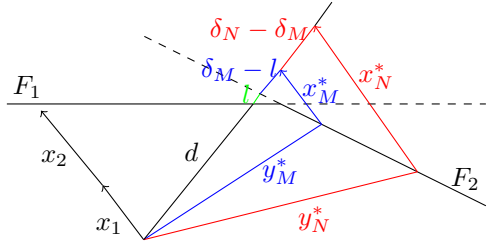
    
    We compare the current value of $r_{X,Y}(d)$ with $r_{F_2}$, its value on $F_2$:
    \begin{equation*}
        r_{F_2} = \frac{D + (\delta_N - l)}{D + (\delta_M - l)} \qquad \text{and} \qquad r_{X,Y}(d) = \frac{D + \delta_N}{D + \delta_M}.
    \end{equation*}
    Since $l > 0$ and $\delta_N > \delta_M$, a simple calculation shows that $r_{X,Y}(d) < r_{F_2}$. Therefore, $r_{X,Y}$ is decreasing during the crossing of a vertex before $v_\pi$ as $\beta$ increases, Lemma~\ref{lemma: crossing before v_pi} holds.
    
    During the crossing of $v_\pi$, $y_N^* = v_\pi$ and $x_N^* = x_{v_\pi}$ with its norm decreasing continuously until $x_N^* = x_1$, while $x_M^*$ will switch to $x_2$ in order to minimize $\delta_M$. This is the same process as described in Lemma~\ref{lemma: crossing v_pi}, so $r_{X,Y}$ also reaches a local minimum.
    
    Because all the results studied so far still hold, then Lemmas~\ref{lemma: x_N^*(d) and x_M^*(d) cst after v_pi}, \ref{lemma: crossing after v_pi} and \ref{lemma: beta > pi} hold too because they rely on those earlier results.  
\end{proof}

We have now established all the lemmas directly involved in the proof of the Maximax Minimax Quotient Theorem, but we still have a few claims of continuity to prove.

\section{Continuity of Extrema}\label{sec:continuity}

In Proposition~\ref{prop: r_X,Y well-defined} (iii) we needed the continuity of $\lambda^*$ to prove it has a minimum and in Lemma~\ref{lemma: x_N^*(d) and x_M^*(d) cst on faces} we used the continuity of $x_N^*$ and $y_N^*$.
In this section we will thus prove the continuity of these two maxima functions relying on the Berge Maximum Theorem \citep{inf_dim_analysis}.

\begin{lemma}\label{lemma:phi continuous}
    Let $X$ and $Y$ be two nonempty polytopes in $\mathbb{R}^n$ with $-X \subset Y$. Then, the set-valued function $\varphi : X \times \mathbb{S} \rightrightarrows Y$ defined as $\varphi(x,d) := Y \cap \big\{ \lambda d - x : \lambda \geq 0 \big\}$ satisfies Definition 17.2 of \citep{inf_dim_analysis}.
\end{lemma}
\begin{proof}
    We define $\Omega := X \times \mathbb{S}$, so that $\varphi : \Omega \rightrightarrows Y$. On the space $\Omega$ we introduce the norm $\|\cdot\|_\Omega$ as $\|(x,d)\|_\Omega = \|x\| + \|d\|$. Since $\|\cdot\|$ is the Euclidean norm, $\|\cdot\|_\Omega$ is a norm on $\Omega$.
    By Definition 17.2 of \citep{inf_dim_analysis}, we need to prove that $\varphi$ is both upper and lower hemicontinuous at all points of $\Omega$.
    
    \vspace{2mm}
    
    First, using Lemma~17.5 of \citep{inf_dim_analysis} we will prove that $\varphi$ is lower hemicontinuous by showing that for an open subset $A$ of $Y$, $\varphi^l(A)$ is open. The lower inverse image of $A$ is defined in \citep{inf_dim_analysis} as
    \begin{align*}
        \varphi^l(A) &:= \big\{ \omega \in \Omega : \varphi(\omega) \cap A \neq \emptyset \big\} \\
        &= \big\{(x,d) \in X \times \mathbb{S} : Y \cap \{ \lambda d - x : \lambda \geq 0 \} \cap A \neq \emptyset\big\} \\
        &= \big\{(x,d) \in X \times \mathbb{S} : \{\lambda d - x : \lambda \geq 0\} \cap A \neq \emptyset\big\},
    \end{align*}
    because $A \subset Y$.
    Let $\omega = (x,d) \in \varphi^l(A)$. Then, there exists $\lambda \geq 0$ such that $\lambda d - x \in A$. Since $A$ is open, there exists $\varepsilon > 0$ such that the ball $B_\varepsilon(\lambda d - x) \subset A$. Now let $\omega_1 = (x_1, d_1) \in \Omega$ and denote $\varepsilon_x := \|x_1 - x\|$ and $\varepsilon_d := \|d_1 - d\|$. Then,
    \begin{align*}
        \|\lambda d_1 - x_1 - (\lambda d - x) \| &= \| \lambda (d_1 - d) - (x_1 - x)\| \leq \lambda \varepsilon_d + \varepsilon_x.
    \end{align*}
    Since $\lambda \geq 0$ is fixed, we can choose $\varepsilon_d$ and $\varepsilon_x$ positive and small enough so that $\lambda \varepsilon_d + \varepsilon_x \leq \varepsilon$.
    Then, we have showed that for all $\omega_1 = (x_1, d_1) \in \Omega$ such that $\|\omega - \omega_1\|_\Omega \leq \min(\varepsilon_d, \varepsilon_x)$, i.e., such that $\|x_1 - x\| \leq \varepsilon_x$ and $\|d_1 - d\| \leq \varepsilon_d$, we have $\lambda d_1 - x_1 \in B_\varepsilon(\lambda d - x) \subset A$, i.e., $\omega_1 \in \varphi^l(A)$. Therefore, $\varphi^l(A)$ is open, and so $\varphi$ is lower hemicontinuous.
    
    \vspace{3mm}
    
    To prove the upper hemicontinuity of $\varphi$, we will use Lemma~17.4 of \citep{inf_dim_analysis} and prove that for a closed subset $A$ of $Y$, the lower inverse image of $A$ is closed. Let $\{\omega_k\}$ be a sequence in $\varphi^l(A)$ converging to $\omega = (x,d) \in \Omega$. We want to prove that the limit $\omega \in \varphi^l(A)$.
    
    For $k \geq 0$, we have $\omega_k = (x_k, d_k)$ and define $\Lambda_k := \big\{ \lambda_k \geq 0 : \lambda_k d_k - x_k \in A \big\} \neq \emptyset$. Since $A$ is a closed subset of the compact set $Y$, then $A$ is compact. Thus $\Lambda_k$ has a minimum and a maximum; we denote them by $\lambda_k^{min}$ and $\lambda_k^{max}$ respectively. 
    
    Since sequences $\{d_k\}$ and $\{x_k\}$ converge, they are bounded. The set $A$ is also bounded, thus sequence $\{\lambda_k^{max}\}$ is bounded. Let $\lambda^{max} := \underset{k\, \geq\, 0}{\sup}\ \lambda_k^{max} > 0$.
    
    For $k \geq 0$, we define segments $S_k := \big\{\lambda d_k - x_k : \lambda \in [0, \lambda^{max}] \big\}$, and $S := \big\{\lambda d - x : \lambda \in [0, \lambda^{max}] \big\}$. These segments are all compact sets.
    We also introduce the sequences $a_k := \lambda_k^{min}d_k - x_k \in A \cap S_k$ and $b_k := \lambda_k^{min}d - x \in S$.
    
    Take $\varepsilon > 0$. Since sequences $\{d_k\}$ and $\{x_k\}$ converge toward $d$ and $x$ respectively, there exists $N \geq 0$ such that for $k \geq N$, we have $\|d_k - d\| \leq \frac{\varepsilon}{2 \lambda^{max}}$ and $\|x_k - x\| \leq \frac{\varepsilon}{2}$. Then, for any $\lambda_k \in [0, \lambda^{max}]$ as
    \begin{equation*}
        \| \lambda_k d_k - x_k - (\lambda_k d - x)\| = \| \lambda_k(d_k - d) - (x_k - x)\| \leq \lambda_k \frac{\varepsilon}{2 \lambda^{max}} + \frac{\varepsilon}{2} \leq \varepsilon.
    \end{equation*}
    Since $\lambda_k^{min} \in [0, \lambda^{max}]$, we have $\|a_k - b_k\| \xrightarrow[k \rightarrow \infty]{} 0$. We define the distance between the sets $A$ and $S$
    \begin{equation*}
        dist(A,S) := \min \big\{ \|a - s_\lambda\| : a \in A,\ s_\lambda \in S\big\}.
    \end{equation*}
    The minimum exists because $A$ and $S$ are both compact and the norm is continuous. Since $a_k \in A$ and $b_k \in S$, we have $dist(A, S) \leq \|a_k - b_k\|$ for all $k \geq 0$. Therefore, $dist(A, S) = 0$. So, $A \cap S \neq \emptyset$, leading to $\omega \in \varphi^l(A)$. Then, $\varphi^l(A)$ is closed and so $\varphi$ is upper hemicontinuous.    
\end{proof}

\vspace{3mm}

\begin{lemma}\label{lemma: lambda continuous}
    Let $X$ and $Y$ be two nonempty polytopes in $\mathbb{R}^n$ with $-X \subset Y$. Then, $\lambda^*(x ,d) := \underset{y\, \in\, Y}{\max} \big\{ \|x + y\| : x + y \in \mathbb{R}^+ d \big\}$ is continuous in $x \in X$ and $d \in \mathbb{S}$.
\end{lemma}
\begin{proof}
    According to Proposition~\ref{prop: r_X,Y well-defined} (ii), whose proof does not rely on the current lemma, $\lambda^*$ is well-defined. We introduce the set-valued function $\varphi : X \times \mathbb{S} \rightrightarrows Y$ defined by $\varphi(x,d) := \big\{ y \in Y : x + y \in \mathbb{R}^+ d \big\} = Y \cap \big( \mathbb{R}^+ d - \{x\} \big)$, where $\mathbb{R}^+ d - \{x\} = \big\{ \lambda d - x : \lambda \geq 0\big\}$. 
    
    We define the graph of $\varphi$ as $\text{Gr}\, \varphi := \big\{ (x,d,y) \in X \times \mathbb{S} \times Y : y \in \varphi(x,d) \big\}$, and the continuous  function $f: \text{Gr}\, \varphi \rightarrow \mathbb{R}^+$ as $f(x,d,y) = \|x + y\|$. Set $X \times \mathbb{S}$ is compact and nonempty. Since $Y$ is compact and $\mathbb{R}^+ d - \{x\}$ is closed, their intersection $\varphi(x,d)$ is compact. 
    Because $-X \subset Y$, for all $x \in X$ we have $-x \in \varphi(x,d)$, so $\varphi(x,d) \neq \emptyset$. According to Lemma~\ref{lemma:phi continuous}, $\varphi$ satisfies Definition $17.2$ of \citep{inf_dim_analysis}. Then, we can apply the Berge Maximum Theorem \citep{inf_dim_analysis} and conclude that $\lambda^*$ is continuous in $x$ and $d$.    
\end{proof}

\vspace{3mm}

\begin{lemma}\label{lemma: continuity of x_N}
    Let $X$ and $Y$ be two nonempty polytopes in $\mathbb{R}^n$ with $-X \subset Y$. Then, the functions $\big( x_N^*, y_N^* \big) (d) = \arg \underset{x\, \in\, X,\, y\, \in\, Y}{\max} \big\{ \|x+y\| : x+y \in \mathbb{R}^+ d \big\}$ are continuous in $d \in \mathbb{S}$.
\end{lemma}
\begin{proof}
    Let $Z := X + Y = \big\{x+y : x \in X,\ y \in Y\big\}$. Then $Z$ is the Minkowski sum of two polytopes, so it is also a polytope \citep{sum_polytopes}. According to Proposition~\ref{prop: r_X,Y well-defined} (i), whose proof does not rely on the current lemma, $\underset{x\, \in\, X,\, y\, \in\, Y}{\max} \big\{ \|x+y\| : x+y \in \mathbb{R}^+ d \big\}$ exists and thus $\underset{z\, \in\, Z}{\max} \big\{ \|z\| : z \in \mathbb{R}^+ d \big\}$ is also well-defined.
    
    Since $-X \subset Y$, for all $x \in X$, $-x \in Y$ and thus $0 \in Z$. Then, $\{0\}$ and $Z$ are two polytopes in $\mathbb{R}^n$ with $\pm 0 \in Z$. According to Lemmma~\ref{lemma: lambda continuous} the function $\lambda^*(0, d) := \underset{z\, \in\, Z}{\max} \big\{ \|z + 0\| : z+0 \in \mathbb{R}^+ d \big\}$ is continuous in $d \in \mathbb{S}$.
    
    Then, we define the continous function $z(d) := \lambda^*(0, d) d \in Z$ for $d \in \mathbb{S}$. Note that $z(d) = \arg\underset{z\, \in\, Z}{\max} \big\{\|z\| : z \in \mathbb{R}^+d \big\} = \big(x_N^*, y_N^*\big)(d)$, so these functions are continuous.
\end{proof}

\section{Illustration}\label{sec:example}

We will now illustrate the Maximax Minimax Quotient Theorem on a simple example. 
We consider polygon $X$ delimited by the vertices $x_1 = (0,-0.5)$ and $x_2 = (0,1)$ in $\mathbb{R}^2$ and polygon $Y$ with vertices $(\pm 1, \pm 2)$ and $(\pm 3, 0)$ as represented on Figure~\ref{fig:sets X Y}.

\begin{figure}[htbp!]
    \centering
    \includegraphics[scale=0.5]{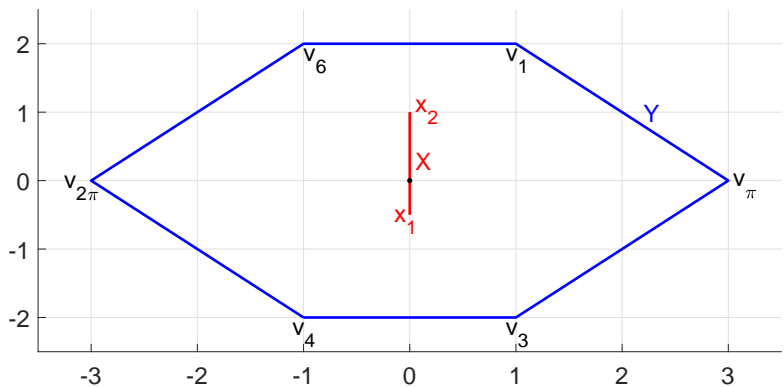} 
    \caption{Illustration of polygons $X$ and $Y$.}
    \label{fig:sets X Y}
\end{figure}

Since $-X \subset Y^\circ$, $\dim X = 1$, $x_2 \neq 0$ and $\dim Y = 2$, the assumptions of the Maximax Minimax Quotient Theorem are satisfied. To illustrate the proof of the theorem, for all $d \in \mathbb{S}$ we define the angle $\beta := \widehat{x_2, d}$ positively oriented clockwise. We also enumerate the vertices in the clockwise direction and we note that $v_2 = v_\pi$ and $v_5 = v_{2\pi}$ as defined in Lemma~\ref{lemma: v_pi and v_2pi}. 
Then, we compute $r_{X,Y}$ for $\beta \in [0, 2\pi)$ as shown on Figure~\ref{fig:r_XY}. The red spikes denote when the ray $d(\beta)$ hits a vertex of $Y$.

\begin{figure}[htbp!]
    \centering
    \includegraphics[scale=0.45]{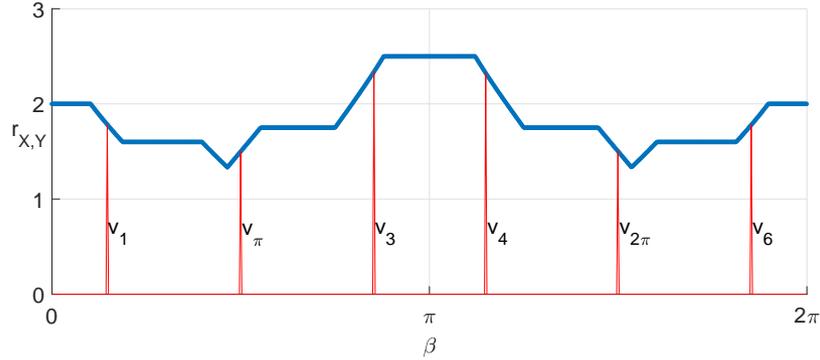} 
    \caption{Graph of $r_{X,Y}$ as a function of $\beta$.}
    \label{fig:r_XY}
\end{figure}

As demonstrated by the Maximax Minimax Quotient Theorem, $r_{X,Y}$ has two local maxima achieved at $\beta = 0$ and $\beta = \pi$. These two values are different because polygon $X$ is not symmetric. Note also that the Maximax Minimax Quotient Theorem does not state that the maximum is \emph{only} reached when $\beta \in \{0, \pi\}$.
Indeed as shown in Figure~\ref{fig:r_XY} and established in Lemma~\ref{lemma:r(d) cst on faces}, $r_{X,Y}$ is constant on the faces of $\partial Y$. Thus, the two local maxima are achieved on the faces $[v_1, v_6]$ and $[v_3, v_4]$.
As proven in Lemma~\ref{lemma: crossing v_pi} and in Lemma~\ref{lemma: beta > pi}, $r_{X,Y}$ reaches a local minimum during the crossing of the vertices $v_\pi$ and $v_{2\pi}$.

A video illustrating the Maximax Minimax Quotient Theorem on a different polytope can be found following the link  \href{https://www.youtube.com/watch?v=rjKzHyDJX40}{\underline{here}} or in the footnote\footnote{\url{https://www.youtube.com/watch?v=rjKzHyDJX40}}.

\section{Conclusion}

In this paper we considered an optimization problem arising from optimal control and pertaining to both fractional programming and max-min programming.
We first justified the existence of the Maximax Minimax Quotient. Then, relying on numerous geometrical arguments and on the continuity of two maxima functions we were able to establish the Maximax Minimax Quotient Theorem.
This result provides an analytical solution to the maximization of a ratio of a maximum and a minimax over two polytopes.
We illustrated our theorem and its proof on a simple example in $\mathbb{R}^2$.
This work filled the theoretical gap left in \citep{SIAM_CT}, and because of our less restrictive assumptions we also open the way for a more general framework than that of \citep{SIAM_CT}.
A possible avenue for future work on this theorem is to study the case where $\dim X > 1$.

\bibliographystyle{abbrv}
\bibliography{references}

\end{document}